\documentclass[12pt]{amsart}

\usepackage{amssymb,latexsym}

\usepackage{enumerate}

\makeatletter

\@namedef{subjclassname@2010}{

  \textup{2010} Mathematics Subject Classification}
\usepackage{amssymb, amsmath,amsthm}
\newtheorem{thm}{Theorem}[section]
\newtheorem{prop}[thm]{Proposition}
\newtheorem{conj}[thm]{Conjecture}
\newtheorem{ques}[thm]{Question}
\newtheorem{cor}[thm]{Corollary}
\newtheorem{lem}[thm]{Lemma}
\theoremstyle{definition}
\newtheorem{rem}[thm]{Remark}
\newtheorem{def1}[thm]{Definition}
\newtheorem{ex}[thm]{Example}

\newcommand{\ra}{\rightarrow}

\newcommand{\mc}{\mathcal}
\newcommand{\mb}{\mathbb}

\newcommand{\llf}{\left\lfloor}
\newcommand{\lla}{\left\langle}
\newcommand{\e}{\varepsilon}
\newcommand{\rrf}{\right\rfloor}
\newcommand{\rra}{\right\rangle}

\renewcommand{\bar}{\overline}

\begin{document}

\baselineskip=17pt
\title[Rigidity Theorems]{Rigidity Theorems for Multiplicative Functions}
\author
{Oleksiy Klurman}
\address{D\'{e}partement de Math\'{e}matiques et de Statistique\\ Universit\'{e} de Montr\'{e}al\\
Montr\'{e}al, Qu\'{e}bec, Canada} 
\address{Department of Mathematics, University College London, Gower Street, London, WC1E 6BT, UK}
\address{Department of Mathematics, KTH Royal Institute of Technology, Stockholm, Sweden}
\email{lklurman@gmail.com}
\author{Alexander P. Mangerel}
\address{Department of Mathematics\\ University of Toronto\\
Toronto, Ontario, Canada}
\email{sacha.mangerel@mail.utoronto.ca}
\maketitle
\begin{abstract}
We establish several results concerning the expected general phenomenon that, given a multiplicative function $f:\mathbb{N}\to\mathbb{C}$, the values of $f(n)$ and $f(n+a)$ are "generally" independent unless $f$ is of a "special" form.\\ 
First, we classify all bounded completely multiplicative functions having uniformly large gaps between its consecutive values. This implies the solution of the following folklore conjecture: for any completely multiplicative function $f:\mathbb{N}\to\mathbb{T}$ we have $$\liminf_{n\to\infty}|f(n+1)-f(n)|=0.$$ \\Second, we settle an old conjecture due to N.G. Chudakov [Actes du ICM ({N}ice, 1970), {T}. 1, p. 487] that states that any completely multiplicative function $f:\mathbb{N}\to\mathbb{C}$ that: a) takes only finitely many values, b) vanishes at only finitely many primes, and c) has bounded discrepancy, is a Dirichlet character. This generalizes previous work of Tao on the Erd\H{o}s Discrepancy Problem. \\
Finally, we show that if many of the binary correlations of a 1-bounded multiplicative function are asymptotically equal to those of a Dirichlet character $\chi$ mod $q$ then $f(n) = \chi'(n)n^{it}$ for all $n$, where $\chi'$ is a Dirichlet character modulo $q$ and $t \in \mb{R}$. This establishes a variant of a conjecture of H. Cohn for multiplicative arithmetic functions. \\
The main ingredients include the work of Tao on logarithmic Elliott conjecture, correlation formulas for  \emph{pretentious} multiplicative functions developed earlier by the first author and Szemeredi's theorem for long arithmetic progressions.
%
%
\end{abstract}
\section{Introduction}
Let $\mathbb{U}$ denote the closed unit disc in $\mb{C}$ and let $\mathbb{T}$ be the unit circle.
Following Granville and Soundararajan, for multiplicative functions $f,g: \mb{N} \ra \mb{U}$ and $x \geq 2$, we define the \emph{distance} between $f$ and $g$ by
\begin{equation*}
\mb{D}(f,g;x) := \left(\sum_{p \leq x} \frac{1-\text{Re}(f(p)\bar{g(p)})}{p}\right)^{\frac{1}{2}}.
\end{equation*}
It is generally believed that the multiplicative structure of an object (a set of integers, say) should not, in principle, interfere with its additive structure, and thus the values of $f(n)$ and $f(n+a)$, where $f$ is multiplicative,  should be roughly independent unless $f$ is ''exceptional'' in some sense. One measure of this "independence" is cancellation in the binary correlations. In fact, we expect that
$$\sum_{n\le x}f(n)\bar{f(n+h)}=o(x)$$
unless $\mb{D}(f,\chi n^{it};x) \ll 1$ for some Dirichlet character $\chi$ and $t\in\mathbb{R}.$
This expectation is in line with  a famous conjecture of Chowla that implies that when $f$ is the Liouville function $\lambda(n) := (-1)^{\Omega(n)}$, where $\Omega(n)$ is the number of prime divisors of $n$, counted with multiplicity,
$$\sum_{n\le x}\lambda(n)\lambda(n+h)=o(x).$$
These conjectures are still widely open in general, though spectacular progress has recently been made as a consequence of the breakthrough of Matom{\"a}ki and Radziwi{\l}{\l}~\cite{MR} and subsequent work of Matom{\"a}ki, Radziwi{\l}{\l} and Tao~\cite{MRT} and more recently by Tao and Ter\"{a}v\"{a}inen~\cite{TaoTerva}. In particular, this led Tao~\cite{Tao} to establish a weighted version of Chowla's conjecture in the form
\[\sum_{n\le x}\frac{\lambda(n)\lambda(n+h)}{n}=o(\log x)\]
for all $h\ge 1.$ More generally, he showed that if  $f:\mathbb{N}\to\mathbb{U}$ is a multiplicative function which is \emph{non-pretentious} in the sense that
\[\sum_{p\le x}\frac{1-\operatorname{Re}(f(p)\chi(p)p^{it})}{p}\ge A\]
for all Dirichlet characters $\chi$ of period at most $A,$ and all real numbers $|t|\le Ax,$ then
\[\left| \sum_{x/\omega<n\le x}\frac{f(n)\bar{f(n+h)}}{n}\right|\le \varepsilon \log{\omega},\]
where $\omega:\mathbb{R}\to [2,\infty)$ is any function tending to infinity with $x$, and the constant $A$ depends at most on $\varepsilon>0. $

This paper is concerned with rigidity problems for multiplicative functions; that is, we seek to understand whether functions can be completely determined by some kind of general hypothesis.\\
The archetype for the problems we shall consider is the famous theorem of Erd\H{o}s \cite{Erd} that states that if $f : \mb{N} \ra \mb{N}$ is a \emph{non-decreasing} multiplicative function then $f(n) = n^k$ for some non-negative integer $k$. Another example of such a rigidity result, first conjectured by K\'{a}tai and solved by Wirsing (and independently by Shao and Tang, see \cite{Wir}), is that if $f: \mb{N} \ra \mb{T}$ is multiplicative and $|f(n+1)-f(n)| \ra 0$ as $n \ra \infty$ then $f(n) := n^{it}$ for some $t \in \mb{R}.$ This type of problems attracted the attention of a number of authors, among whom K\'{a}tai, Hildebrand, Phong, Elliott, Wirsing and others. See, for example,
the survey paper~\cite{Kat2} which includes an extensive list of the related references. We shall discuss a variant of K\'{a}tai's problem, well-known to experts, in the next subsection.
\subsection{On Consecutive Values of Unimodular Multiplicative Functions}
Wirsing's theorem addresses the case in which a unimodular multiplicative function eventually has small gaps between its consecutive values. At the other extreme, we may consider the problem of classifying those unimodular multiplicative functions $f$ such that these gaps are \emph{never} small for large $n$. One consequence of our results is the solution to the following folklore conjecture, showing that such $f$ \emph{cannot} be completely multiplicative.
\begin{thm} \label{FOLK}
For any completely multiplicative $f:\mathbb{N}\to\mathbb{T}$ $$\liminf_{n\to \infty}|f(n+1)-f(n)|=0.$$
\end{thm}
\begin{rem}
The proof of Theorem \ref{FOLK} is effective, in principle. With additional effort, we could determine the growth rate of the sequence $\{n_j\}_j$ on which $|f(n_j+1)-f(n_j)| \ra 0$, in terms of $\e$. For instance, our proof relies on Tao's theorem, Theorem \ref{TAOBINTHM} below, which can be made effective (see Remark 1.4 of \cite{Tao}) as well as effective version of Szemeredi's theorem due to Gowers. See Remark \ref{DIRAPP} following the proof of Theorem \ref{FOLK} below for a further discussion regarding effectivity. For the sake of clarity and space we have chosen to omit this calculation.
\end{rem}
We shall in fact determine precise conditions under which the conclusion of Theorem \ref{FOLK} fails to hold for general multiplicative (but not necessarily completely multiplicative) functions.  
\begin{thm} \label{EPSTHM}
Fix $\e > 0$. Suppose $f : \mb{N} \ra \mb{T}$ is a multiplicative function such that $|f(n+1)-f(n)| \geq \e$ for all sufficiently large $n$. Then there are (minimal) integers $k,q = O_{\e}(1)$ and a real number $t = O_{\e}( 1)$ such that: \\
a) there exists a completely multiplicative function $g$ such that $\mb{D}(f,gn^{it};x) \ll_{\e} 1$ and for which there exists a Dirichlet character $\chi$ modulo $q$ such that $g(n)^k = \chi(n)$ for each $(n,q) = 1$; \\
b) $1$ is not a limit point of $\{f((2q)^{l})(2q)^{-ilt}\}_{l\ge 1}.$
\end{thm}
We note that some problem at the prime 2 must occur. Indeed, the multiplicative function $f(n) := (-1)^{n-1}$ (defined by taking $f(2^k) = -1$ and $f(p^k) = 1$ for all $p \geq 3$ and all $k \geq 1$) is an example of such a function for all $\e \in [0,2]$.
\begin{rem} \label{GENH}
Our proof can also be modified to treat the case in which the shift 1 is replaced by any fixed $h \in \mb{N}$ (in which the obstruction at powers of 2 is replaced by an obstruction depending on $h$ in Theorem \ref{EPSTHM}). 
\end{rem}
Theorem \ref{FOLK} has a number of arithmetic consequences. For example, it implies the following.
\begin{cor} \label{ARITH}
For $k\geq 1$ let $A_1,\ldots,A_k$ be disjoint sets of primes and let $q_1,\ldots,q_k$ be coprime integers. For each $j$, let $\Omega_{A_j}(n)$ denote the number of prime factors of $n$ belonging to $A_j$, counted with multiplicity. Then there are infinitely many $n \in \mb{N}$ such that for all $j$, $\Omega_{A_j}(n+1) \equiv \Omega_{A_j}(n) \text{ (mod $q_j$)}$.
\end{cor}
This follows immediately from Theorem \ref{FOLK} by taking the completely multiplicative function $f$ defined by $f(p) = e(1/q_j)$ whenever $p \in A_j$ for each $j$, and $f(p) = 1$ for all other primes $p$. When $k = 1$ and $A_1$ is the set of all primes, Corollary \ref{ARITH} follows (trivially) from work of Heath-Brown and later Hildebrand related to the Erd\H{o}s-Mirsky problem on the infinitude of $n$ with $\Omega(n+1) = \Omega(n)$ (see Theorem 5 in \cite{Hil2}). On the other hand, Corollary \ref{ARITH} permits us to choose \emph{multiple} sets in a \emph{flexible} manner. \\ 
The proof of Theorem \ref{EPSTHM} can also be modified to settle the following conjecture due to K\'{a}tai and Subbarao (see \cite{KS1}, \cite{KS2}, as well as Wirsing's paper \cite{WIRKS} for partial results in this direction).\footnote{We thank Imre K\'{a}tai for drawing our attention to this conjecture.}
\begin{thm}[K\'{a}tai-Subbarao Conjecture] \label{KSCONJ}
Let $f: \mb{N} \ra \mb{T}$ be a completely multiplicative function. Then the set of limit points of $\{f(n)\bar{f(n+1)}\}_n$ is $\mb{T}$, unless there exists a minimal positive integer $k$ and a real number $t$ such that if $h(n) := f(n)n^{-it}$ then $h(n)^k = 1$ for all $n$. In particular, in the former case we have that $\{f(n)\bar{f(n+1)}\}_n$ is dense in $\mb{T}$. In the latter case, the set of limit points of $\{f(n)\bar{f(n+1)}\}_n$ is equal to the set of all $k$th roots of unity.
\end{thm}
Writing a unimodular, completely multiplicative function $f$ in the form $e^{2\pi i u}$, where $u$ is a completely additive function taking values in $\mb{R}/\mb{Z}$, Theorem \ref{KSCONJ} yields the following consequence.
\begin{cor} \label{KSCOR}
Let $u: \mb{N} \ra \mb{R}/\mb{Z}$ be a completely additive function. Then, unless there exists a positive integer $k$ such that $ku(n) = t \log n \ (\text{mod } 1)$ for some $t \in \mb{R}$, the sequence $\{u(n)-u(n+1)\}_n$ is dense in $\mb{R}/\mb{Z}$.
\end{cor}
\noindent We prove Theorem \ref{EPSTHM} (and similarly, Theorem \ref{KSCONJ}) in several steps. \\
First, we show that the hypothesis $|f(n+1)-f(n)| \geq \e$ for large $n$ implies (via Tao's work on the logarithmically averaged version of Elliott's conjecture) that there is a completely multiplicative function $g$ taking values on roots of unity of bounded order (depending on $\e$) and some $t \in \mb{R}$ such that $\mb{D}(f,gn^{it};x) \ll_{\e} 1$; by considering sufficiently large $n$, we may assume that $t = 0$. \\
The thread of the remainder of the proof is to show that if condition b) in Theorem \ref{EPSTHM} fails then the hypothesis $|f(n+1)-f(n)| \geq \e$ must fail for some sufficiently large integer $n$. To this end, we will show the existence of a ''structured'' set $S$ on which $\sum_{n \leq x \atop n \in S} \frac{1}{n}\left|f(n)-f(n+1)\right|^2$ is small enough (in terms of $\e$) that the trivial lower bound $\e^2 \sum_{n \leq x \atop n \in S} \frac{1}{n}$ implied by the above-mentioned hypothesis cannot hold. To do this, we must ensure that the (log-averaged) correlation sums in $f(n)\bar{f(An+1)}$ on $S$ are large, for some suitable positive integer $A$. Assuming b) fails, we choose $A= (2q)^T$, where $T$ is selected such that $f((2q)^T)$ is close to $1$.\\
To simplify our work, we make two observations. First, if we choose $S$ to belong to a set on which $g(n) = g((2q)^Tn+1)$ then we can avoid the discrete oscillation in argument contributed by $g$ and focus only on the oscillation of $f$ due to the 1-pretentious function $F = f\bar{g}$. It is then sufficient to control the binary correlation sum in $F$. Second, the correlation sum in $F$ is easily computed when $S$ is chosen to be a long arithmetic progression. Moreover, the sum is large if we assume additionally that for each $n \in S$, the least prime factor of $n((2q)^Tn+1)$ is large (in terms of $\e$). \\
To guarantee that a suitable such arithmetic progression $S$ exists, we show that the set consisting of "presieved" integers $n$ for which $g(n) = g((2q)^Tn+1)$ and $P^-(n((2q)^Tn+1)) > N$ has positive upper density for each fixed $N$. $S$ can then be chosen by applying the effective version of Szemer\'{e}di's theorem, due to Gowers \cite{Gow}.   \\
See Section \ref{BIGREM} for a corollary of Theorem \ref{EPSTHM}, motivated by considerations in equidistribution theory. \\
It is natural to try to extend Theorem \ref{FOLK} to completely multiplicative functions taking values in $\mb{U}$ more generally, rather than just in $\mb{T}$. In this more general context, however, one can plainly find examples of functions with uniformly large gaps. Indeed, if $\chi$ is a character modulo a prime $p$ with order $p-1$ then $\chi$ is necessarily injective; as such, $\chi(n) \neq \chi(n+1)$, for all $n$, a condition that is equivalent to the large gaps hypothesis (with $\e$ sufficiently small) since $\chi$ takes only finitely many values. \\
One might guess that characters are unique in this respect, and that all examples of functions satisfying the hypothesis ought to be character-like, in some sense. We give a precise version of such a statement below. \\
Given a set of primes $S$, let $\lla S \rra$ denote the \emph{monoid} (i.e., semigroup containing 1) generated by $S$. By an element $a$ in $\lla S\rra$ we mean a (finite) positive integer generated by products of elements in $S$. We say that $S$ is a \emph{thin} set if
\begin{equation}\label{THINDEF}
\sum_{p\in S}\frac{1}{p}<\infty.
\end{equation}
\begin{thm} \label{DISC}
Let $f: \mb{N} \ra \mb{U}$ be a completely multiplicative function for which $\liminf_{n \ra \infty} |f(n+1)-f(n)| > 0$. Then either: \\
a) $|f(2)| < 1$; or \\
b) there is a prime $p$, minimal positive integers $k,l,M$, a real number $t$ and a Dirichlet character $\chi$ modulo $q = p^l$ such that:
\begin{enumerate}[(i)]
\item $(f(n)n^{-it})^M = \chi(n)^M = 1$ for all $n$ coprime to $p$; 
\item $f$ is pretentious to a function $gn^{it}$, and $g(n)^k = \chi(n)$ whenever $p\nmid n$; 
\item the function $h(n) := f(n)\bar{g(n)}n^{-it}$ is supported on a thin set $S$ of primes that either consists only of primes congruent to $1$ modulo $q$, or else if $c,d \in \lla S\rra$, $\text{ord}(h(c)),\text{ord}(h(d)) = M$ and $c \equiv d (q)$ then $h(c) = h(d)$.
\end{enumerate}
\end{thm}
Theorem \ref{DISC} implies that a function with uniformly large gaps between its consecutive values must behave ``like'' a character in the sense that for most primes, a fixed power of $f$ takes the same values as a character, and generic integers in a given residue class modulo $q$ that are composed of the remaining primes are all assigned the same value. In a sense, $f$ is ``generically periodic''. \\
In Section \ref{DISCSEC}, we give two minimal examples of functions $f$ satisfying the properties (i) - (iii) and each of the two cases in (iv), and additionally verifying $f(n) \neq f(n+1)$ for all $n$.  
\subsection{On a Conjecture of Chudakov}
The Polymath5 project reduced the Erd\H{o}s discrepancy problem (EDP), now a theorem due to Tao~ \cite{EDP}, to the statement about multiplicative functions. In particular, Tao~\cite{EDP} established that
for any {\it completely multiplicative} function $f:\mathbb{N}\to\{-1,1\}$ 
\begin{equation}\label{completeunb}
\sup_{x\ge 1}\left|\sum_{n\le x}f(n)\right|=\infty.
\end{equation}
Recall that a Dirichlet character is a completely multiplicative function $\chi : \mb{N} \ra \mb{C}$ for which there is a positive integer $q$ such that $\chi(n+q) = \chi(n)$ for all $n$ and $\chi(n) = 0$ whenever $(n,q) > 1.$ It is clear that 
\[\left|\sum_{n\le x}\chi(n)\right|\le q=O(1),\]
and consequently Dirichlet characters provide near-counterexample to the EDP.\\
It was suggested in~\cite{EDP} that such an obstruction is essentially the only one. We confirm this guess by proving a conjecture of Chudakov \cite{Chu} from 1956 (see also~\cite{Chu1}).\footnote{We are grateful to Sergey Konyagin and Terence Tao for communicating to us this problem.}
\begin{thm}[Chudakov's Conjecture]\label{CHUDAKOV}
Let $f : \mb{N} \ra \mb{C}$ be a completely multiplicative function such that 
\begin{enumerate} [(i)]
\item $|\{f(n) : n \in \mb{N}\}| < \infty$;\\
\item $|\{p : f(p) = 0\}| < \infty$;\\
\item there is an $\alpha \in \mb{C}$ such that $\sum_{n \leq x} f(n) = \alpha x + O(1)$ as $x \ra \infty$. \\
\end{enumerate}
Then $f$ is a Dirichlet character. 
\end{thm}	
In 1964, Glazkov \cite{Gla} settled the case $\alpha \neq 0$ via analytic means (in which case $f$ must be the principal character modulo some $q$ and $\alpha = \phi(q)/q$). We shall thus only consider the case $\alpha = 0$, which has remained open since.\\
We note that Theorem~\ref{CHUDAKOV} implies~\eqref{completeunb} and provides an \emph{analytic} characterization of Dirichlet characters. This can be compared with the result of Sark\H{o}zy~\cite{Sar}, who showed that if $\chi$ is any completely multiplicative function satisfying a non-trivial linear recurrence relation then $\chi$ is a Dirichlet character.
\subsection{On Cohn's Conjecture}
We next consider the following general question: can a 1-bounded multiplicative function $f$ be completely determined by its binary correlations. That is, suppose $f,g: \mb{N} \ra \mb{U}$ are multiplicative functions such that for some set $S \subset \mb{N}$,
\begin{equation} \label{BINEQ}
x^{-1}\sum_{n \leq x} f(n)\bar{f(n+h)} \sim x^{-1}\sum_{n \leq x} g(n)\bar{g(n+h)}
\end{equation}
for all $h \in S$. If $S$ is sufficiently large, must it be true that $f (n)= g(n)n^{it}$? \\
This question is difficult to answer in general: in particular, if $f$ is non-pretentious then we do not even know whether the quantity on the left side of \eqref{BINEQ} is $o(1)$ for any fixed $h$. We have a slight edge when considering pretentious functions (in which case the work \cite{Klu} of the first author is of relevance), though in this case the question is still difficult to answer in full generality. \\
To motivate the precise case of the above question that we shall address, we state the following open problem due to H. Cohn (see Section 11 of \cite{Mon}).
\begin{conj}[H. Cohn]\label{ORIG}
Let $p$ be an odd prime and let $f: \mb{F}_{p} \ra \mb{C}$ be a map satisfying $f(0) = 0$, $f(1) = 1$ and $|f(a)| = 1$ for all $a \in \mb{F}_{p}$. Assume moreover that for each $h \in \mb{F}_p$, we have
\begin{equation} \label{COHNCOND}
\sum_{a\in \mb{F}_p} f(a)\bar{f(a+h)} = \begin{cases} -1 &\text{ if $h \neq 0$} \\ p-1 &\text{ otherwise.} \end{cases}
\end{equation}
Then $f$ is a multiplicative character on $\mb{F}_p$.
\end{conj}
A simple calculation shows that every multiplicative character on $\mb{F}_p$ satisfies these hypotheses. Thus, Cohn's conjecture is asking whether a function on $\mb{F}_p$ is essentially determined by the values of its binary correlations. \\
This problem is still open in the finite field setting (for partial results, see Bir\'{o}'s paper \cite{Bir} and Kurlberg's paper \cite{Kur}). We shall focus on the following \emph{approximate} version of Conjecture \ref{ORIG}, suited to multiplicative arithmetic functions.
\begin{ques}\label{COHN}
Let $q \geq 1$ be odd, and let $H$ be a large positive integer. Let $f : \mb{N} \ra \mb{U}$ be multiplicative, and suppose there is a primitive Dirichlet character $\chi$ with conductor $q$ such that for each $1 \leq h \leq H$,
\begin{align}
\sum_{n \leq x}f(n)\overline{f(n+h)}& = (1+o(1))\sum_{n \leq x} \chi(n)\overline{\chi(n+h)}, \label{ast2}
\end{align}
as $x \ra \infty$. Must $f$ be a Dirichlet character modulo $q$?
\end{ques}
In this form, one might expect the conjecture to be false for a slight technical reason: in principle, the perturbation $o(1)$ in the hypothesis might allow both $f$ and a perturbed version of it to both satisfy \eqref{ast2}. Indeed, this turns out to be true, but we are able to completely determine the way in which $f$ can be perturbed.
\begin{prop} \label{APPROXMULT}
Let $q$ be an odd positive integer and let $\chi$ be a primitive character modulo $q$. Let $H \geq q$, and suppose $f:\mb{N} \ra \mb{U}$ is a multiplicative function taking values in the unit disc that satisfies \eqref{ast2} for all $1 \leq h \leq H$. Then $f(n) = \chi'(n)n^{it}$, where $\chi'$ is also primitive with conductor $q$, and $t \in \mb{R}$. 
\end{prop}
We remark that Proposition \ref{APPROXMULT} implies that Conjecture \ref{ORIG} holds for all \emph{multiplicative} functions $f: \mb{F}_p \ra \mb{C}$ (extended by periodicity modulo $p$ to all of $\mb{N}$). 
\section*{Acknowledgments}
\noindent We would like to thank John Friedlander and Andrew Granville for all their advice and encouragement. We are grateful to Andrew Granville and Terence Tao for their valuable comments on an earlier version of the present paper. We thank Mei-Chu Chang for her interest in the results of this paper. We are indebted to Sergey Konyagin and Terence Tao for introducing us to Chudakov's conjecture. We also thank Imre K\'{a}tai for referring us to his conjecture with Subbarao. Finally, we would like to thank both the Mathematical Sciences Research Institute in Berkeley, California, as well as the Fields Institute for Research in the Mathematical Sciences in Toronto for providing excellent working conditions.
\section{On Consecutive Values of Unimodular Multiplicative Functions} \label{SECFN}
Let $f: \mb{N} \ra \mb{T}$ be a multiplicative function for which there is an $\e > 0$ such that 
\begin{equation}\label{EHYPOTHESIS}
|f(n+1)-f(n)| \geq \e \text{ for all sufficiently large $n$.}
\end{equation}
Note that this hypothesis implies that the sequence $\{f(n)\bar{f(n+1)}\}_n$ is not dense in $\mb{T}$, as it avoids an $\Omega(\e)$-neighbourhood about 1. As such, it must have a large discrepancy, a fact that can be expressed using the following.
\begin{lem}[Weighted Erd\H{o}s-Tur\'{a}n Inequality] \label{ETW}
Let $\{w_n\}_n$ be a sequence of positive real numbers and let $W_N := \sum_{1 \leq n \leq N} w_n$. Define the weighted discrepancy $D_N(\{\theta_n\}_n,\{w_n\}_n)$ of a sequence $\{\theta_n\}_n \subset [0,1]$ by
\begin{equation*}
D_N(\{\theta_n\}_n,\{w_n\}_n) := \sup_{0 \leq a < b \leq 1} \left|W_N^{-1}\sum_{n \leq N \atop \theta_n \in [a,b)} w_n - (b-a)\right|.
\end{equation*}
Then for any positive integer $m \in \mb{N}$, 
\begin{equation*}
D_N(\{\theta_n\}_n,\{w_n\}_n) \leq \frac{6}{m+1} + \frac{4}{\pi}\sum_{h \leq m}\left(\frac{1}{h}-\frac{1}{m+1}\right) \left|W_N^{-1} \sum_{n \leq N} w_ne(h\theta_n)\right|.
\end{equation*}
\end{lem}
The proof is a basic generalization of the proof of the usual Erd\H{o}s-Tur\'{a}n inequality, but as such an extension is not readily found in the literature, we give a short proof here, based on the proof of Theorem 2.5 in \cite{KuN}.
\begin{proof}
For $0 \leq x \leq 1$, define $\Delta_N(x) := W_N^{-1}\sum_{n \leq N \atop \theta_n \in [0,x)} w_n - x$, and assume for the moment that $\{\theta_n\}_n$ satisfies
\begin{equation} \label{ZEROINT}
\int_0^1 \Delta_N(x) dx = 0.
\end{equation}
Let $e(t) := e^{2\pi i t}$, for $t \in \mb{R}$. Extend $\Delta_N(x)$ to a map on $\mb{R}$ by periodicity. Observe that for each $h \in \mb{N}$
\begin{align*}
\int_0^1 \Delta_N(x)e(hx) dx &= W_N^{-1}\sum_{n \leq N} w_n \int_{\theta_n}^1 e(hx) dx - \int_0^1 xe(hx) dx \\
&= \frac{1}{2\pi i h} \left(W_N^{-1}\sum_{n \leq N} w_n(1-e(h\theta_n))\right) - \frac{1}{2\pi i h} \\
&= -\frac{1}{2\pi i h} W_N^{-1} \sum_{n \leq N} w_ne(h\theta_n) =: -\frac{S_h}{2\pi i h}.
\end{align*}
Put $\tilde{\Delta}_{N,m}(x) := \Delta_N \ast F_m(x)$, where $F_m$ is the $m$th order F\'{e}jer kernel. Owing to \eqref{ZEROINT} and Lebesgue invariance,
\begin{align*}
\tilde{\Delta}_{N,m}(a) &= \frac{1}{m+1}\int_0^1 \Delta_N(x+a) \left(\frac{\sin((m+1)\pi x)}{\sin \pi x}\right)^2 dx \\
&= \sum_{1 \leq |h| \leq m}\left(1-\frac{|h|}{m+1}\right)e(-ha)\int_0^1 \Delta_N(x) e(hx) dx \\
&= -\frac{1}{2\pi i}\sum_{1 \leq |h| \leq m} \left(1-\frac{|h|}{m+1}\right)\frac{S_h}{h}.
\end{align*}
As such, for any $a \in [0,1]$, we have the uniform upper bound
\begin{equation*}
\sup_{0 \leq a \leq 1} |\tilde{\Delta}_{N,m}(a)| \leq \frac{1}{\pi} \sum_{1 \leq h \leq m}\left(1-\frac{h}{m+1}\right) \frac{|S_h|}{h} = \frac{1}{\pi W_N} \sum_{1 \leq h \leq m} \left(\frac{1}{h}-\frac{1}{m+1}\right)\left|\sum_{n \leq N} w_ne(h\theta_n)\right|.
\end{equation*}
The remainder of the proof under the assumption of \eqref{ZEROINT} is precisely the same as that given on p.113-114 of \cite{KuN}, up to equation (2.40) there, giving
\begin{equation} \label{INVAR}
D_N \leq \frac{6}{m+1} + \frac{4}{\pi}\sum_{1 \leq h \leq m} \frac{1}{h} \left(W_N^{-1}\left|\sum_{n \leq N} w_ne(h\theta_n)\right|\right).
\end{equation}
It thus remains to show that upon replacing $\{\theta_n\}_n$ by a sequence $\{\theta_n + c\}_n$, we can ensure that \eqref{ZEROINT} is satisfied, as both sides of \eqref{INVAR} are invariant under this translation. The proof of this fact follows \emph{mutatis mutandis} from the argument in \cite{KuN}.
\end{proof}
\begin{lem} \label{LARGELOG}
Suppose $f$ satisfies \eqref{EHYPOTHESIS}. Let $N := 12\pi(\llf 2/\e\rrf + 1)$. There is an $x_0 = x_0(\e)$ such that for each $x \geq x_0$ there is some $1 \leq k \leq N$ such that
\begin{equation*}
\left|\sum_{n \leq x} \frac{(f(n)\bar{f(n+1)})^k}{n} \right| \geq \delta\log x,
\end{equation*}
with $\delta = \e/(18 \log N)$.
\end{lem}
\begin{proof}
Suppose for a contradiction that this fails for a given $x$ and all $1 \leq k \leq N$. By Lemma \ref{ETW} with $w_n := \frac{1}{n}$ and $\theta_n := \frac{1}{2\pi}\text{arg}(f(n)\bar{f(n+1)})$ (with the principal branch), we have
\begin{align*}
\left|\sum_{n \leq x \atop \theta_n \in [\e/2\pi,1)} \frac{1}{n} - (1-\e/2\pi)H_x\right| &\leq \frac{6H_x}{N+1} + \frac{4}{\pi}\sum_{1 \leq k \leq N} \frac{1}{k} \left|\sum_{n \leq x} \frac{(f(n)\bar{f(n+1)})^k}{n} \right| \\
&< \frac{6H_x}{N+1} + \frac{4}{\pi} H_N\delta \log x,
\end{align*}
where $H_t := \sum_{n \leq t}  \frac{1}{n}$. On the other hand, as $|f(n)\bar{f(n+1)}-1| \geq \e$, it is obvious that $\theta_n \in [\e/2\pi,1)$ for all $n \leq x$. Hence, 
\begin{equation*}
\left|\sum_{n \leq x \atop \theta_n \in [\e/2\pi,1)} \frac{1}{n}-\left(1-\frac{\e}{2\pi}\right)H_x\right| \geq \frac{\e}{2\pi} H_x.
\end{equation*}
As such, we have
\begin{equation*}
\left(\e-\frac{12\pi}{N+1}\right)H_x < 8 H_N \delta \log x = 8 H_N\delta H_x + O(\log N).
\end{equation*}
Dividing by $H_x$, and picking $x$ sufficiently large in terms of $\e$ alone, we have
\begin{equation*}
\delta > \frac{1}{9H_N}\left(\e-\frac{12\pi}{N+1}\right) = \frac{\e}{18 \log N},
\end{equation*}
which is a patent contradiction.
\end{proof}
The following lemma is an effective version of a result due to Elliott (see Lemma 4.2 of \cite{Klu}).
\begin{lem} \label{ELLLEM}
Let $a,A > 1$, and let $\{x_j\}_j$ be an increasing sequence of positive real numbers such that for all $n \in \mb{N}$, $x_n < x_{n+1} \leq x_n^a$. Let $f: \mb{N} \ra \mb{U}$ be a multiplicative function. Suppose moreover that for each $j$ there is a Dirichlet character $\chi_j$ modulo $q_j \leq A$ and a real number $t_j = O(x_j)$ such that $\mb{D}(f,\chi_j n^{it_j};x_j) \ll 1$. Then there is a Dirichlet character $\chi$ and a real number $t$ such that $\mb{D}(f,\chi n^{it};x) \ll_{a,A} 1$, where $t = t_j + O_A\left(\frac{1}{\log x_j}\right)$ for $j$ sufficiently large.
\end{lem}
\begin{proof}
Note that the order of each character $\chi_j$ is bounded in terms of $A$. Let $k=k(A)$ be a sufficiently large integer such that $\chi_j^k$ is principal modulo $q_j$, for all $j$. It follows from the triangle inequality that
\begin{equation*}
\mb{D}(f^k,n^{ikt_j};x_j) \leq k \mb{D}(f,\chi_jn^{it_j};x_j) + O_A(1) \ll_A 1.
\end{equation*}
Since $\mb{D}$ is monotone, it follows that for $m,l \in \mb{N}$, $m > l$ we have
\begin{equation*}
\mb{D}(1,n^{ik(t_m-t_l)};x_l) \leq \mb{D}(f^k,n^{ikt_m};x_l) + \mb{D}(f^k,n^{ikt_l};x_l) \ll_A 1.
\end{equation*}
This is equivalent to the statement that $$\log \left|\zeta\left(1+\frac{1}{\log x_n} + ik(t_m-t_n)\right)\right| = \log_2 x_n + O_A(1),$$ and hence that $|t_m-t_n| \ll_A \frac{1}{\log x_n}$. Thus, $\{t_j\}_j$ is a Cauchy sequence that converges to a real number $t$, and hence $|t-t_n| \ll_A \frac{1}{\log x_n}$.\\
Now let $x$ be an arbitrary, large real number. Choose $m$ such that $x_m < x \leq x_{m+1}$. Then
\begin{align*}
\mb{D}(f,\chi_m n^{it};x) &= \mb{D}(f,\chi_m n^{it};x_m) + O\left(\sum_{x_m < p \leq x} \frac{1}{p}\right) \\
&= \mb{D}(f,\chi_m n^{it};x_m) + O\left(\log(\log x_{m+1}/\log x_m)\right) \\
&= \mb{D}(f,\chi_m n^{it_m}; x_m) + O_{a,A}(1) \ll_{a,A} 1.
\end{align*}
Lastly, observe that for $m,l \in \mb{N}$, $m > l$,
\begin{equation} \label{CLOSECHARS}
\mb{D}(\chi_m,\chi_l;x_l) \leq \mb{D}(f,\chi_m n^{it};x_l) + \mb{D}(f,\chi_ln^{it};x_l) \ll_{a,A} 1.
\end{equation}
Note that as $\chi_m$ and $\chi_l$ are primitive, unless $\chi_m \neq \chi_l$ it follows that $\chi_m \bar{\chi_l}$ is non-principal, and hence $\log L(1+1/\log x,\chi_m \bar{\chi_l}) \ll_A 1$. On the other hand, \eqref{CLOSECHARS} implies that $$\log L\left(1+\frac{1}{\log x},\chi_m\bar{\chi_l}\right) = \log_2 x + O_{a,A}(1)$$ (a proof of this standard estimate follows e.g., from Lemma 3.4 in \cite{LaM}, and Mertens' theorem).  Hence, $\chi_m = \chi_l$ for all $l < m$. Letting $\chi$ denote this common character proves the claim.
\end{proof}
Throughout this paper we shall appeal to the following result, due to Tao (this is a consequence of Corollary 1.5 in \cite{Tao}).
\begin{thm}[Tao]\label{TAOBINTHM}
Let $f_1,f_2$ be 1-bounded multiplicative functions, such that $f_1$ is non-pretentious. Let $a_1,b_1,a_2,b_2$ be non-negative integers such that $a_1b_2 \neq a_2b_1$. Then
\begin{equation*}
\sum_{n \leq x} \frac{f_1(a_1n+b_1)f_2(a_2n+b_2)}{n} = o(\log x).
\end{equation*}
\end{thm}
\begin{prop}\label{PRETk}
Suppose $f$ satisfies \eqref{EHYPOTHESIS}. Then there are positive integers $k,q = O_{\e}(1)$, a primitive Dirichlet character $\chi$ modulo $q$ and a real number $t$ such that $\mb{D}(f^k,\chi n^{it};x) \ll_{\e} 1$.
\end{prop}
\begin{proof}
By Lemma \ref{LARGELOG}, there is some minimal $k = O_{\e}(1)$ such that for a set of Lebesgue measure $\gg_{\e} x$ in $[1,x]$, we have
\begin{equation} \label{LARGE}
\left|\sum_{n \leq x} \frac{(f(n)\bar{f(n+1)})^k}{n} \right| \gg_{\e} \log x.
\end{equation}
Fix this $k$ and denote by $S_k = \{x_j\}_j$ the associated set of integers $x$ satisfying \eqref{LARGE}. By Corollary 1.5 of \cite{Tao}, it follows that for each $j$ sufficiently large there are primitive Dirichlet characters $\chi_j$ of modulus $O_{\e}(1)$ and $t_j \ll_{\e} x_j $ such that $\mb{D}(f^k,\chi_j n^{it_j};x_j) \ll_{\e} 1$ as $j \ra \infty$. By passing to an infinite subsequence if necessary, we may assume that $x_j < x_{j+1} \leq x_j^2$ (if only finitely many such $x_j$ existed then $S_k$ could not have positive density). 
By Lemma \ref{ELLLEM}, it follows that for some $t= O_{\e}(1)$ we have $\mb{D}(f^k,\chi n^{it};x) \ll 1$, as claimed.
\end{proof}
We will assume henceforth that $k$ is chosen \emph{minimally} in Proposition \ref{PRETk}.
\begin{lem} \label{PRETG}
Suppose $f$ satisfies \eqref{EHYPOTHESIS}. Let $m$ and $q$ be, respectively, the order and modulus of the character $\chi$ with $\mb{D}(f^k,\chi n^{it};x) \ll_{\e} 1$. There is a completely multiplicative function $g$ taking values in the set of $mk$th roots of unity such that: i) $\mb{D}(f,gn^{it/k};x) \ll_{\e} 1$; and ii) $g(n)^k = \chi(n)$ whenever $(n,q) = 1$.
\end{lem}
\begin{proof}
i) First, since $m|\phi(q)$ and $q = O_{\e}(1)$, it follows that $m = O_{\e}(1)$ as well. Now, we use an idea of Granville and Soundararajan (see Section 2.1.6 in \cite{GrSo}). Let $g$ be the completely multiplicative function defined such that $g(p)$ is the closest $mk$th root of unity to $f(p)p^{-it/k}$. Then $$\|\text{arg}(f(p)\bar{g(p)}p^{-it/k})\| \leq \pi/(mk)$$ for each prime $p$, where $\|t\| := \min\{\{t\},1-\{t\}\}$. By the triangle inequality, we have $$\mb{D}(f^{mk},n^{itm};x) \leq O_{\e}(1) + \mb{D}((f^k\bar{\chi}n^{-it})^m,1;x) \ll_{\e} 1+ m\mb{D}(f^k,\chi n^{it};x) \ll_{\e} 1,$$ again using $q,m = O_{\e}(1)$. On the other hand, we have $1-\cos \theta \leq 1-\cos(mk\theta)$ for all $|\theta| \leq \pi/mk$. Hence, as $g^{mk} = 1$,
\begin{equation*}
\mb{D}(f,g n^{it/k};x) \leq \mb{D}(f^{mk},g^{mk}n^{itm};x) \ll_{\e} 1.
\end{equation*}
ii) Appealing once again to the triangle inequality, it follows that $$\mb{D}(g^k,\chi; x) \leq \mb{D}(f^k,g^kn^{it};x) + \mb{D}(f^k,\chi n^{it}) \ll_{\e} 1.$$ Now, for each $0 \leq l \leq k-1$ define $S_l$ to be the set of primes $p$ such that $g(p)^k\bar{\chi}(p) = e(l/k)$. Observe that for each $l$,
\begin{equation*}
\sum_{p \leq x \atop p \in S_l} \frac{1-\text{Re}(e(l/k))}{p} \leq \mb{D}(g^k,\chi;x)^2 \ll_{\e} 1,
\end{equation*}
and for $l \neq 0$ this implies that $S_l$ is thin, and in particular $\sum_{p \in S_l} p^{-1} \ll_{\e} 1$. Consequently, we must have
\begin{equation*}
\sum_{p \leq x \atop p \in S_0} \frac{1}{p} = \log_2 x + O_{\e}(1),
\end{equation*}
and thus, except on the thin set $S := \bigcup_{l \neq 0} S_l$, we have $g(p)^k = \chi(p)$.
%
%
Adjusting $g(p)$ along this set (which does not affect the condition $\mb{D}(f,g n^{it/k};x) \ll_{\e} 1$) we may have $g(p)^k = \chi(p)$ for $p \nmid q$, and thus $g(n)^k = \chi(n)$ for all $n$ coprime to $q$ by complete multiplicativity. 
\end{proof}
\begin{rem}\label{REMT}
Note that we can assume without loss of generality that $t = 0$, since if $t \neq 0$ and $n > 2|t|/\e$ we have
\begin{equation*}
|f(n)n^{-it}-f(n+1)(n+1)^{-it}| \geq |f(n)-f(n+1)| - |(1+1/n)^{-it}-1| \geq \e/2.
\end{equation*}
We thus henceforth assume that $\mb{D}(f,g;x) \ll_{\e} 1$, where $g(n)^k = \chi(n)$ with $\chi$ a character mod $q$ with exponent $m$, $k$ is minimal and $k,m = O_{\e}(1)$, and $(n,q) = 1$.
\end{rem}
We will need the following version of Szemeredi's theorem, due to Gowers~\cite{Gow}.
\begin{lem}\label{SZEM}
Let $\mc{A} \subseteq [1,x]$ with $|\mc{A}| = \delta x$. Then $\mc{A}$ contains a progression of length $\mc{L} \asymp \log_2 \left(\frac{\log_3 x}{\log(1/\delta)}\right)$.
\end{lem}
\begin{proof}
This follows from Theorem 1.3 of \cite{Gow}, where the statement is that if $\delta \geq (\log_2 x)^{-c_k}$ with $c_k := 2^{-2^k + 9}$ then $\mc{A}$ contains a progression of length $k$.
\end{proof}
Fix $T$ to be a positive integer to be chosen later. We write $P^-(n)$ to denote the least prime factor of $n$, as usual.
\begin{lem} \label{ROUGHLOGAP}
Let $q \geq 1$ and let $N \geq 2$ and let $a$ be a residue class modulo $q$ such that $(a((2q)^Ta+1),q) = 1$. Then
\begin{equation*}
\sum_{n \leq x \atop n \equiv a (q), P^-(n((2q)^Tn+1)) > N} \frac{1}{n} = \frac{3}{4}\frac{\log x}{q} \prod_{3 \leq p \leq N \atop p \nmid q} \left(1-\frac{2}{p}\right) + O\left(4^{\pi(N)}\right).
\end{equation*}
\end{lem}
\begin{proof}
Write $\sum^{\ast}$ here to mean that $n$ is restricted such that $P^-(n((2q)^Tn+1)) > N$, and put $P_{N} := \prod_{p \leq N} p$. 
Then
\begin{align*}
&\sideset{}{^{\ast}}\sum_{n \leq x \atop n \equiv a(q)} \frac{1}{n} = 
\sum_{d_1,d_2|P_N \atop (d_1,d_2) = 1, (2,d_2) = 1} \mu(d_1)\mu(d_2) \sideset{}{^{\ast}}\sum_{n \leq x}\frac{1_{n \equiv a(q)} 1_{n \equiv 0(d_1)} 1_{n \equiv -\bar{2q}^T (d_2)}}{n}
\end{align*}
Note that $(d_1,q) = (d_2,q) = 1$, as $(d_1,q)|a$ and $(d_2,q)|((2q)^Ta+1)$, and the only common divisor that $a$ or $(2q)^Ta+1$ shares with $q$ is 1; thus, $\bar{q}$ is well-defined. Hence, by the Chinese Remainder Theorem, for some residue class $b$ modulo $qd_1d_2$ we have
\begin{align*}
&\sideset{}{^{\ast}}\sum_{n \leq x \atop n \equiv a(q)} \frac{1}{n} = \sum_{d_1,d_2|P_N} \mu(d_1)\mu(d_2)1_{(d_1,d_2) = 1}1_{(d_1d_2,q) = 1}1_{(2,d_2) = 1} \left(\sum_{n \leq x/qd_1d_2} \frac{1}{b+qd_1d_2n} \right)\\
&= \sum_{d_1,d_2|P_N} \mu(d_1)\mu(d_2) 1_{(d_1,d_2) = 1}1_{(d_1d_2,q) = 1}1_{(2,d_2) = 1}\left(\frac{1}{qd_1d_2}\sum_{n \leq x/qd_1d_2} \frac{1}{n} + O(1)\right) \\
&= \frac{(\log x)}{q} \sum_{d_1|P_N \atop (d_1,q) = 1} \frac{\mu(d_1)}{d_1} \sum_{d_2|P_N \atop (d_2,2q) = 1, (d_1,d_2) = 1} \frac{\mu(d_2)}{d_2} + O\left(4^{\pi(N)}\right),
\end{align*}
as the number of pairs of divisors of $P_N$ is at most $4^{\pi(N)}$. Evaluating the product in $d_2$, with $(d_1,q) = 1$, we get
\begin{equation*}
\sum_{d_2|P_N \atop (d_2,2qd_1) = 1} \frac{\mu(d_2)}{d_2} = \prod_{p | P_N \atop p \nmid 2qd_1} \left(1-\frac{1}{p}\right),
\end{equation*}
and hence summing over all $d_1$, we get
\begin{equation*}
\prod_{p|P_N \atop p \nmid q} \left(1-\frac{1}{p}\right) \sum_{d_1 |P_N \atop (d_1,q) = 1} \frac{\mu(d_1)}{d_1}\prod_{p|2d_1}\left(1-\frac{1}{p}\right)^{-1} = \frac{3}{4}\prod_{3 \leq p \leq N \atop p \nmid q} \left(1-\frac{2}{p}\right),
\end{equation*}
where the factor of $3/4$ appears because
\begin{align*}
&\prod_{p|p_N \atop p \nmid q} \left(1-\frac{1}{p}\right)\sum_{d_1 |P_N \atop (d_1,q) = 1} \frac{\mu(d_1)}{d_1}\prod_{p|2d_1}\left(1-\frac{1}{p}\right)^{-1} \\
&= \prod_{p|p_N \atop p \nmid q} \left(1-\frac{1}{p}\right)\sum_{d_1|P_N/2 \atop (d_1,q) = 1} \frac{\mu(d_1)}{d_1} \prod_{p|2d_1}\left(1-\frac{1}{p}\right)^{-1} - \frac{1}{2}\sum_{d_1|P_N/2 \atop (d_1,q) = 1} \frac{\mu(d_1)}{d_1}\prod_{p|d_1} \left(1-\frac{1}{p}\right)^{-1} \\
&= \frac{3}{4}\prod_{3 \leq p \leq N \atop p \nmid q} \left(1-\frac{1}{p}\right)\sum_{d_1|P_N/2 \atop (d_1,q) = 1} \frac{\mu(d_1)}{d_1}\prod_{p|d_1} \left(1-\frac{1}{p}\right)^{-1} = \frac{3}{4}\prod_{3 \leq p \leq N \atop p \nmid q} \left(1-\frac{2}{p}\right).
\end{align*}
The claim easily follows.
\end{proof}
\begin{lem} \label{POSDENS}
Let $g$ be the completely multiplicative function constructed in Lemma \ref{PRETG}, and let $m$ and $k$ be as in Lemma~\ref{PRETG}. Let $N \geq q$ be a large, fixed constant. Let $T$ be a fixed non-negative integer. Let $\mc{A}_{g,T}(N)$ be the set of $n \in \mb{N}$ such that $g(n) = g((2q)^Tn+ 1)$ and $P^-(n((2q)^Tn+1)) > N$.  Then $\mc{A}_{g,T}(N)$ has positive logarithmic density $\delta_{g,T}$ given by
\begin{equation*}
\delta_{g,T} := \left(\frac{3}{4mk}+ o(1)\right)\prod_{3 \leq p\leq N} \left(1-\frac{2}{p}\right).
\end{equation*}
\end{lem}
\begin{proof}
The key idea is that  $m$ and $k$ are chosen minimally in Lemma \ref{PRETG}, which forces $g^s$ to be non-pretentious for all $1 \leq s < k$ whenever $k \geq 2.$ This allows us to apply Tao's result to get cancellation of the binary correlations. Put $l := mk$, where $g^l = 1$ identically, and $g^k = \tilde{\chi}$, where $\tilde{\chi}(n) = \chi(n)$ if $(n,q) = 1$, and $\tilde{\chi}(n) = 1$ otherwise. Let $g'(n) = g(n)1_{P^-(n) > N}$, which is still completely multiplicative. Observe that 
\begin{align*}
\sum_{n \leq x \atop n \in \mc{A}_{g,T}(N)} \frac{1}{n} &= \sum_{n \leq x \atop P^-(n((2q)^Tn+1)) > N} \frac{1}{n}\prod_{1 \leq j \leq l-1} \frac{1-\zeta^{j}g(n)\bar{g((2q)^Tn+1)}}{1-\zeta^j} \\
&= \frac{1}{l}\sum_{n \leq x \atop P^-(n((2q)^Tn+1)) > N} \frac{1}{n}\prod_{1 \leq j \leq l-1} (1-\zeta^{-j}g(n)\bar{g((2q)^Tn+1)}),
\end{align*}
since $$\prod_{1 \leq j \leq l-1} (x-\zeta^j) = \sum_{0 \leq j \leq l-1} x^j =: P(x),$$ and $P(1) = l$. Expanding the right side gives
\begin{align*}
&\frac{1}{l}\sum_{S \subseteq \{1,\ldots,l-1\}} (-1)^{|S|} \zeta^{\Sigma(S)} \sum_{n \leq x \atop P^-(n((2q)^Tn+1)) > N} \frac{(g'(n)\bar{g'((2q)^Tn+1)})^{|S|}}{n},
\end{align*}
where $\Sigma(S) := \sum_{j \in S} j$ (this is defined to be 0 when $S$ is empty). Now, as $(g')^s$ is non-pretentious for each $1 \leq s \leq k-1$, it follows that $\chi^r(g')^s$ is also non-pretentious for each $0 \leq r \leq m-1$ and $1 \leq s \leq k-1$. As such, by Theorem \ref{TAOBINTHM} (with $f_1 = (g')^s$ and $f_2 = (\bar{g}')^s$) we have
\begin{align*}
&\sum_{n \leq x \atop P^-(n((2q)^Tn+1)) > N} \frac{(g'(n)\bar{g'((2q)^Tn+1)})^s(\chi(n)\bar{\chi((2q)^Tn+1)})^r}{n} \\
&= \sum_{n \leq x} \frac{(g'(n)\bar{g'((2q)^Tn+1)})^s(\chi(n)\bar{\chi((2q)^Tn+1)})^r}{n} = o(\log x),
\end{align*}
whenever $s \neq 0$. Hence,
\begin{equation} \label{EXPAND}
\sum_{n \leq x \atop n \in \mc{A}_{g,T}(N)} \frac{1}{n} = \frac{1}{l}\sum_{0 \leq r \leq m-1} (-1)^{rm}\sum_{S \subset \{1,\ldots,l-1\} \atop |S| = rm}\zeta^{\Sigma(S)} \sum_{n \leq x \atop P^-(n((2q)^Tn+1)) > N} \frac{(\tilde{\chi}(n)\bar{\tilde{\chi}((2q)^Tn+1)})^r}{n} + o(\log x).
\end{equation}
When $1 \leq r \leq m-1$ observe that $\chi((2q)^Tn+1) = 1$ for all $n$; hence, for $x$ large in terms of $N$, Lemma \ref{ROUGHLOGAP} gives
\begin{align*}
&\sum_{n \leq x \atop P^-(n((2q)^Tn+1)) > N} \frac{(\tilde{\chi}(n)\bar{\tilde{\chi}((2q)^Tn+1)})^r}{n} = \sum_{a(q)}\chi(a)^r \sum_{n \leq x \atop n \equiv a (q), P^-(n((2q)^Tn+1)) > N} \frac{1}{n} \\
&= \frac{3}{4}\prod_{3 \leq p \leq N \atop p \nmid q}\left(1-\frac{2}{p}\right) (\log x) \sum_{a(q)} \chi(a)^r + O(4^{\pi(N)}) = o(\log x).
\end{align*}
The only remaining term is $r = 0$, in which case we get the term 
\begin{equation*}
\sum_{n \leq x \atop P^-(n((2q)^Tn+1)) > N} \frac{1}{n} = \frac{3}{4} \prod_{3 \leq p \leq N} \left(1-\frac{2}{p}\right) \log x + O(4^{\pi(N)})
\end{equation*}
by Lemma \ref{ROUGHLOGAP} (with 1 in place of $q$ in the congruence condition). The claim now follows.
\end{proof}
\begin{lem}\label{CORRELAP}
Given $J \geq 1$ and $a,d \in \mb{N}$, let $Q = \{a + jd : 1 \leq j \leq J-1\} \subset [1,x]$ be an arithmetic progression. Let $f : \mb{N} \ra \mb{U}$ be a multiplicative function such that $f(p^k) = 0$ whenever $p|P_{N}$. Let $L_1(n) := dn+a$ and $L_2(n) := (2q)^Tdn + (2q)^Ta+1$. Then 
\begin{equation*}
\sum_{n \in Q} \frac{f(n)\bar{f((2q)^Tn+1)}}{n} = \left(\sum_{n \in Q} \frac{1}{n}\right)\left(\prod_{N < p \leq x} M_p(f;L_1,L_2) + O\left(\mb{D}(f,1;N,x) + \frac{T+1}{\log_2 x}\right)\right).
\end{equation*}
where 
for each prime $p$,
\begin{equation*}
M_p(f;L_1,L_2) := \lim_{x \ra \infty} x^{-1} \sum_{\nu_1,\nu_2 \geq 0} f(p^{\nu_1}) \bar{f(p^{\nu_2})} \sum_{n \leq x \atop p^{\nu_1}||L_1(n), p^{\nu_2}||L_2(n)} 1.
\end{equation*}
\end{lem}
\begin{rem}\label{UNIF}
Theorem 1.3 in \cite{Klu} does not specify that the estimates there are uniform in some range of the size of the coefficients of the polynomials considered there. It is clear, however, that if we assume that $a,d \leq A$ with $A \ra \infty$ as $x \ra \infty$ then, in the case of linear forms under consideration here, that theorem does still hold either with the additional error term $(\log A)/\log x$, or with $\mb{D}(f,1;x)$ replaced by $\mb{D}(f,1;Ax)$. See, for instance, Lemma 2.2 of \cite{KlM} for an example of this calculation.
\end{rem}
\begin{proof}
By Theorem 1.3 in \cite{Klu}, this is precisely
\begin{align*}
&\sum_{n \leq J} f(L_1(n))\bar{f(L_2(n))} + O(1) \\
&= J\left(\prod_{N < p \leq (2q)^TdJ}M_p(f;L_1,L_2) + O\left(\mb{D}(1,f; N, (2q)^TdJ) + \frac{1}{\log_2 x}\right)\right),
\end{align*}
and as $dJ \leq x$ by assumption and $T$ is fixed, we can replace the product above with
\begin{equation*}
\prod_{N < p \leq x} M_p(f;L_1,L_2) + O\left(\mb{D}(1,f;N,x) + \frac{T+1}{\log_2 x}\right).
\end{equation*}
The claim now follows by partial summation.
\end{proof}
\begin{proof}[Proof of Theorems \ref{FOLK} and \ref{EPSTHM}]
We suppose that $|f(n+1)-f(n)| \geq \varepsilon$ for all sufficiently large $n$. Let $N$ be a sufficiently large positive integer (to be specified solely in terms of $\varepsilon$). Combining Lemma \ref{LARGELOG} and Proposition \ref{PRETk}, we may assume that there are (minimal) positive integers $k,q \ll_{\e} 1$ and a Dirichlet character $\chi$ of modulus $q$ such that $\mb{D}(f^k,\chi n^{it};x) \ll_{\e} 1$. According to Remark \ref{REMT}, we can assume furthermore that $t = 0$. Let $m$ be the order of $\chi$ as an element of the dual group of $\left(\mb{Z}/q\mb{Z}\right)^{\ast}$ (which, as $m|\phi(q)$, must also be $O_{\e}(1)$). By Lemma \ref{PRETG} there is a completely multiplicative function $g$ such that $g(n)^k = \chi(n)$ whenever $(n,q) = 1$ and $g^{km} = 1$ identically. \\
We now assume for the sake of contradiction that $1$ is a limit point of $\{f((2q)^l)\}_l$. 
Then we choose $T$ in such a way that $|f((2q)^{T})-1| < C\e^2$, where $C := \min\{1/10,\varepsilon^{-2}\}$. 
Since a set of integers with positive logarithmic density also has positive upper density, by Lemma \ref{POSDENS}, as soon as $N$ is sufficiently large (solely in terms of $\e$) we have $$\limsup_{ x \ra \infty} x^{-1}|\mc{A}_{g,T}(N)\cap [1,x]| > 0,$$ where we recall that $$\mc{A}_{g,T}(N) = \{n \in \mb{N}: P^-\left(n((2q)^Tn+1)\right) > N, g(n) = g((2q)^Tn+1)\}.$$ Thus, assuming $x$ lies on a suitable subsequence, we can write $|\mc{A}_{g,T}(N) \cap [1,x]| = \delta x$, with $\delta > 0$. Hence, by Lemma \ref{SZEM}, we can extract an arithmetic progression $Q$ with $|Q| \ra \infty$ as $x \ra \infty$ along this subsequence. In particular, writing $f = gF$, where $F$ is 1-pretentious, we have
\begin{equation*}
\sum_{n \in Q} \frac{F(n)\bar{F((2q)^Tn+1)}}{n} = \left(\sum_{n \in Q} \frac{1}{n}\right)\left(\prod_{N < p \leq x} M_p(F;L_1,L_2) + O\left(\mb{D}(F,1; N, x)+ \frac{T+1}{\log_2 x}\right)\right),
\end{equation*}
by Lemma \ref{CORRELAP} (using the notation there). We can replace the product above by 
\begin{align*}
\prod_{N < p \leq x} M_p(F;L_1,L_2) &= \prod_{p > N} \left(1-\frac{1}{p}\right)\left(\sum_{k \geq 0} \frac{F(p^k)}{p^k}\right) + O\left(\frac{1}{\log N}\right) \\
&= 1 + O\left(\mb{D}(F,1;N,x) + \frac{1}{\log N}\right),
\end{align*}
and so
\begin{equation} \label{SHARP}
\sum_{n \in Q} \frac{F(n)\bar{F((2q)^Tn+1)}}{n} = \left(\sum_{n \in Q} \frac{1}{n}\right)\left(1+O\left(\mb{D}(F,1;N,x) + \frac{1}{\log N}\right)\right).
\end{equation}
Now, observe that for $n \in Q$, $f(n)\bar{f((2q)^Tn+1)} = F_N(n)\bar{F_N((2q)^Tn+1)}$, where $F_N = F1_{P^-(n) > N}$. It follows that
\begin{align} \label{LOWBOUNDEPS}
\e^2 \sum_{n \in Q} \frac{1}{n} &\leq \sum_{n \in Q} \frac{1}{n}|f((2q)^Tn+1)-f((2q)^Tn)|^2 = \sum_{n \in Q} \frac{1}{n}|f((2q)^Tn+1)-f((2q)^T)f(n)|^2.
\end{align}
Given our choice of $T$,
we have 
\begin{align*}
|f((2q)^T)f(n)-f((2q)^Tn+1)|^2 &\leq |f(n)-f((2q)^Tn+1)|^2 + 4C\e^2 + C^2\e^4 \\
&\leq |f(n)-f((2q)^Tn+1)|^2 + 5C\e^2.
\end{align*}
Hence, 
\begin{align*}
&\sum_{n \in Q} \frac{1}{n}|f((2q)^Tn+1)-f((2q)^T)f(n)|^2 \leq \sum_{n \in Q} \frac{1}{n}|1-f(n)\bar{f((2q)^Tn+1)}|^2 + 5C\e^2 \left(\sum_{n \in Q} \frac{1}{n}\right) \\
&= \sum_{n \in Q} \frac{1}{n}|1-F_N(n)\bar{F_N((2q)^Tn+1)}|^2 + 5C\e^2 \left(\sum_{n \in Q} \frac{1}{n}\right)\\
&= (2+5C\e^2)\sum_{n \in Q} \frac{1}{n}-2\sum_{n \in Q} \text{Re}\left(\frac{F_N(n)\bar{F_N((2q)^Tn+1)}}{n}\right),
\end{align*}
and thus by \eqref{LOWBOUNDEPS}
\begin{equation}
\text{Re}\left(\sum_{n \in Q} \frac{F_N(n)\bar{F_N((2q)^Tn+1)}}{n}\right) \leq \left(1-\frac{1}{2}(1-5C)\e^2\right) \sum_{n \in Q} \frac{1}{n} \leq \left(1-\frac{\e^2}{4}\right)\sum_{n \in Q} \frac{1}{n}. \label{ZETA}
\end{equation}
On the other hand, \eqref{SHARP} yields the inequality
\begin{equation*}
\sum_{n \in Q} \frac{1}{n} \leq \left(1-\frac{\e^2}{5}\right) \sum_{n \in Q} \frac{1}{n},
\end{equation*}
when $N$ is chosen sufficiently large in terms of $\varepsilon$. This contradiction completes the proof.
\end{proof}
\begin{rem}\label{DIRAPP}
Note that for general multiplicative functions, the choice of $T$ depending on $\e$ is ineffective in the proof of Theorem \ref{EPSTHM}. However, when $f$ is completely multiplicative this choice \emph{can} be made effective. Indeed, if $f(2q) = e(\theta)$, where $\theta \notin \mb{Q}$ then $T\theta \notin \mb{Q}$. Thus, e.g., by Dirichlet's theorem in Diophantine approximation, $T$ can be chosen with $|f(2q)^T - 1| < C\e^2$ as above with $T = O_{\e}(1)$ with an effective dependence of $T$ on $\e$. Otherwise, if $\theta \in \mb{Q}$ then we need only choose $T$ such that $T\theta = 0$ on $\mb{R}/\mb{Z}$. This choice of $T$ depends on $q$, which depends in an effective way on $\e$ (see Remark 1.4 of \cite{Tao}).
\end{rem}
\begin{proof}[Proof of Theorem \ref{KSCONJ}]
Our proof will largely follow the proof of Theorem \ref{EPSTHM}, so we shall leave some details to the reader. \\
We shall assume that for each $l \geq 1$ there exists some $n \in \mb{N}$ such that $f(n)^l \neq 1$. Let $z \in \mb{T}$. Assume that $|f(n)-z f(n+1)| \geq \e$ for all sufficiently large $n$. By the proof of Theorem \ref{EPSTHM}, there exists a function $g$ taking values on roots of unity of bounded order, minimal integers $k,m = O_{\e}(1)$ and a modulus $q = O_{\e}(1)$ such that $g(n)^k = \chi(n)$ for all $(n,q) = 1$ and $g^{mk} = 1$, and $t \in \mb{R}$ such that $\mb{D}(f,gn^{it};x) \ll_{\e} 1$. As before, we may assume that $t = 0$. \\
Let $r,l \in \mb{N}$ be chosen such that $|f(r)^lf(2q) - \bar{z}| < C\e^2$, with $C > 0$ as in the proof of Theorem \ref{EPSTHM}. To find such $r$ and $l$, we consider two cases. First, if there is some $r$ such that $f(r) = e(\alpha)$, where $\alpha \notin \mb{Q}$ then we can use Dirichlet's theorem as before to find $l$ for which $f(r)^l$ approximates (a rational approximation of) $\bar{zf(2q)}$. On the other hand, if $f(n)$ always has rational argument, then we can choose $r$ and $l$ so that $f(r) = e(a/b)$, where $al/b$ approximates the argument of $z\bar{f(2q)}$ to error $O(\e^2)$ as above. Now, in the same way as was done in the proof of Theorem \ref{EPSTHM}, we may extract a long arithmetic progression $Q$ from the set of integers $n\leq x$ such that $g(n) = g(2qr^ln + 1)$ and $P^-(n(2qr^ln + 1)) > N$, where $N$ is sufficiently large in terms of $\e$. For $n\in Q$, 
\begin{align*}
\e^2 \sum_{n \in Q} \frac{1}{n} &\leq \sum_{n \in Q} \frac{1}{n}|f(2qr^ln)-zf(2qr^ln+1)|^2 \\
&= 2\left(\sum_{n \in Q} \frac{1}{n} - \text{Re}\left(\bar{z}f(2q)f(r)^l\sum_{n \in Q} \frac{f(n)\bar{f(2qr^ln+1)}}{n} \right)\right) \\
&\leq 2\left(\sum_{n \in Q} \frac{1}{n} - \text{Re}\left(\sum_{n \in Q} \frac{f(n)\bar{f(2qr^ln + 1)}}{n} \right)\right) + 2C\e^2\sum_{n \in Q} \frac{1}{n}.
\end{align*}
Choosing $N$ sufficiently large so that the correlation sum is close to $\sum_{n \in Q} \frac{1}{n}$, as in the proof of Theorem \ref{EPSTHM} allows us to prove the required contradiction implying that $\liminf_{n \ra \infty} |f(n)-zf(n+1)| = 0$, and thus that $z$ is a limit point of $f$. Since $z \in \mb{T}$ was arbitrary, this completes the proof.
\end{proof}
\begin{rem}
It is clear that, in case there is some minimal $k$ such that $f(n)^k = 1$ for all $n$ then the set of limit points of $\{f(n)\bar{f(n+1)}\}_n$ is necessarily finite, and can only consist of $k$th roots of unity. The argument of Lemma \ref{POSDENS} shows that in this case, $f(n)\bar{f(n+1)}$ takes each $k$th root of unity on a set of positive upper density. We leave the details of this to the reader. A consequence of this is that when $f(n)^k = 1$ for all $n$, the set of limit points is equal to the set of all $k$th roots of unity, as claimed by Theorem \ref{KSCONJ}.
\end{rem}
\section{A Strengthening of Elliott's Conjecture and Theorem \ref{EPSTHM}}\label{BIGREM}
There is a natural corollary of Theorem \ref{EPSTHM}, applicable for multiplicative functions that are not completely multiplicative (in contrast to Theorem \ref{KSCONJ}) which requires some discussion. To this end, we introduce the following definition.
\begin{def1} \label{HNP}
A multiplicative function $f : \mb{N} \ra \mb{U}$ is \emph{highly non-pretentious} if, for any fixed $N \in \mb{N}$ and any completely multiplicative function $g : \mb{N} \ra \mb{T}$ satisfying $g^N = 1$, we have $\inf_{|t| \leq x} \mb{D}(f,gn^{it};x) \ra \infty$ as $x \ra \infty$.
Conversely, if $f$ is not highly non-pretentious then we shall call it \emph{pseudo-pretentious}.
\end{def1}
Clearly, any highly non-pretentious function is non-pretentious in the sense mentioned in the introduction.
On the other hand, the M\"{o}bius function, for instance, is non-pretentious in the usual sense but it is pseudo-pretentious. \\
The following is an easy consequence of Hal\'{a}sz' theorem and the Weyl criterion (see, for instance, Section 2.1.6 of \cite{GrSo}).
\begin{prop} \label{EQUIF}
Let $f: \mb{N} \ra \mb{T}$ be completely multiplicative. Then $\{f(n)\}_n$ is equidistributed if, and only if, $f$ is highly non-pretentious. 
\end{prop}
In fact, this can be generalized to multiplicative functions, provided one takes care to avoid the case that for some $l \in \mb{N}$, $f(2^k)^l = -1$ for all $k$. \\
There is an heuristic correspondence between Hal\'{a}sz' theorem and Elliott's conjecture. That is, provided that $f$ does not additionally correlate with Dirichlet characters, if $f$ has small mean value then it has small binary correlations. Note that highly non-pretentious functions necessarily do not correlate with Dirichlet characters. Therefore, by analogy, we propose the following conjecture, motivated by Proposition \ref{EQUIF}.
\begin{conj}\label{VDCMULT}
If $f: \mb{N} \ra \mb{T}$ is a highly non-pretentious multiplicative function then $\{f(n)\bar{f(n+h)}\}_n$ is equidistributed in $\mathbb{T}$ for all $h \in \mb{N}$.
\end{conj}
The hypothesis and corresponding conclusion are both stronger than those that play a role in Elliott's conjecture. However, we still think it worthwhile to consider this variant of Elliott's conjecture as well.\\	
The proof of Theorem \ref{EPSTHM} (in light of Remark \ref{GENH}) provides a weak result in the direction of Conjecture \ref{VDCMULT}. 
\begin{cor} \label{CORGAP}
If $f : \mb{N} \ra \mb{T}$ is a highly non-pretentious multiplicative function then for any $z \in \mb{T}$ we have $\liminf_{n \ra \infty} |f(n+h)-zf(n)| = 0$. 
In particular, $\{f(n)\bar{f(n+h)}\}_n$ is dense in $\mathbb{T}.$
\end{cor}
Since the arguments are similar to those in the proof of Theorem \ref{EPSTHM}, we leave the details to the reader. 
Note that, in contrast to Theorem \ref{FOLK}, Corollary \ref{CORGAP} does not hold for arbitrary completely multiplicative functions
 (cf. Theorem \ref{KSCONJ}). Indeed, the proof of Theorem \ref{EPSTHM} with $z$ in place of $1$ fails because the left side of \eqref{ZETA} in the former case is no longer close to 1 if, say, $z = -1$. 
The salient point here is that \emph{highly non-pretentiousness} condition implies that there is no completely multiplicative function $g$ (as in i) of Theorem \ref{EPSTHM}) for which $\mb{D}(f,gn^{it};x)$ is uniformly bounded for $|t| \leq x$.
\section{The Gaps Problem for 1-bounded Multiplicative Functions: Proof of Theorem \ref{DISC}} \label{DISCSEC}
In this section we show that if $f: \mb{N} \ra \mb{U}$ is completely multiplicative and $|f(n+1)-f(n)| \geq \e$ for all $n$ sufficiently large then $f$ satisfies the properties listed in Theorem \ref{DISC}. We begin the proof of this result by establishing the pseudo-pretentiousness of $f$, identifying the modulus of the corresponding character as a prime power and showing that $f$ must have finite order (wherever it does not vanish). For this we first need the following simple lemma.
\begin{lem} \label{SELECTPRIME}
Let $f,g:\mb{N} \ra \mb{U}$ be multiplicative, and let $\chi$ be a character modulo $q$ such that $g^k = \tilde{\chi}$ and $\mb{D}(f,g;x) \ll 1$. Let $J \geq 1$. Then for any $\delta > 0$ and any coprime residue class $a$ modulo $q$ there exists a prime $p \equiv a(q)$ such that $|f(p)^J - g(p)^J| < \delta$. 
\end{lem}
\begin{proof}
Suppose otherwise, and let $x$ be large. Obviously, if $p$ satisfies $|f(p)^J - g(p)^J| \geq \delta$ then 
$$\delta J^{-1}\leq J^{-1}|f(p)-g(p)|\left|\sum_{0 \leq j \leq J-1} f(p)^jg(p)^{J-1-j}\right| \leq |f(p)-g(p)|.$$
Thus, we have
\begin{align*}
\frac{1}{\phi(q)} \log_2 x + O(1) &= \sum_{p \leq x \atop |f(p)^J-g(p)^J| \geq \delta} \frac{1_{p \equiv a(q)}}{p} \leq \sum_{p \leq x \atop |f(p)-g(p)| \geq \delta J^{-1}} \frac{1_{p \equiv a(q)}}{p}\\
&\leq \left(\frac{J}{\delta}\right)^2\sum_{p \leq x \atop p \equiv a(q)} \frac{|f(p)-g(p)|^2}{p}.
\end{align*}
Expanding the square gives $$|f(p)|^2 + |g(p)|^2 - 2\text{Re}(f(p)\bar{g(p)}) \leq 2(1-\text{Re}(f(p)\bar{g(p)})).$$ Inserting this expression into the above bound gives
\begin{equation*}
\frac{1}{\phi(q)} \log_2 x +O(1) \leq 2\left(\frac{J}{\delta}\right)^2\sum_{p \leq x} \frac{1-\text{Re}(f(p)\bar{g(p)})}{p} = 2\left(\frac{J}{\delta}\right)^2 \mb{D}(f,g;x)^2 \ll \left(\frac{J}{\delta}\right)^2.
\end{equation*}
As the quantity on the right is independent of $x$, we may take $x \ra \infty$, which yields a contradiction.
\end{proof}
\begin{lem} \label{REDUCTIONS}
Suppose $f$ is as in Theorem \ref{DISC}, such that $|f(2)| = 1$. The following holds: \\
a) $|f(p)| < 1$ for exactly one odd prime $p$; \\
b) there are $k,l\in\mathbb{N}$ and $t\in\mathbb{R}$ such that $\mathbb{D}(f, gn^{it}; \infty)<\infty$, where $g: \mb{N} \ra \mb{T}$ is a completely multiplicative function satisfying $g^k = \tilde{\chi}$, and $\chi$ is a Dirichlet character modulo $p^l$;\\
c) there is $M\in\mathbb{N}$, such that $mk\vert M$, such that for all $(n,p)=1$ we  have $(f(n)n^{-it})^M = 1$.
\end{lem}
\begin{proof}
a) Observe, that Theorem \ref{FOLK} implies that there exists a prime $p$ with $|f(p)| < 1$. Suppose there is more than one such prime. Let $p_1 \neq p_2$ be such that $|f(p_j)| < 1$ for $j = 1,2,$ choose $k_j \in \mb{N}$ such that $|f(p_j)|^{k_j} < \e/2$ for each $j$ and $p_2^{k_2} \equiv 1 (p_1^{k_1})$. Writing $p_2^{k_2} = 1+mp_1^{k_1}$, we have
\begin{equation*}
\e \leq |f(1+mp_1^{k_1}) - f(mp_1^{k_1})| \leq |f(p_2)|^{k_2} + |f(m)||f(p_1)|^{k_1} < \e/2 + \e/2 = \e,
\end{equation*}
a contradiction. \\
b) By the same argument leading to the proof of Theorem \ref{FOLK}, we see that $f$ is $gn^{it}$-pretentious, and $g^k = \tilde{\chi}$. Replacing $f(n)$ by $f(n)n^{-it}$ we may assume that $t = 0$.
We claim that $\chi$ must have modulus a power of $p$. Indeed, if not then write $q = p^l\ell^{\nu}c$, where $p\ell \nmid c$, and $\ell$ is a prime distinct from $p$. Choose $A = (2\ell^{\nu}c)^J$, where $J$ is such that $f(2\ell^{\nu}c)^J = 1 + O(\e^2)$. It then suffices to check that there are infinitely many solutions to the equation $g(n) = g(An+1)$. Tracing through the proof of Lemma 2.9, we note that in order to reach its conclusion, we needed that
\begin{equation} \label{REDMOD}
\sum_{n \leq x \atop P^-(n(An+1)) > N} \frac{(\chi(n)\bar{\chi(An+1)})^r}{n} = o(\log x)
\end{equation}
for each $r \neq 0$. Splitting this sum into residue classes modulo $q$, it suffices to show that $\sum_{a(q)} \chi^r(a)\bar{\chi^r(Aa+1)} = 0$. By the Chinese Remainder Theorem we can factor this complete sum as $$\left(\sum_{a(p^l)} \chi_{p^l}^r(a)\bar{\chi_{p^l}^r(aA+1)}\right) \sum_{a' (c\ell^{\nu})} (\chi'(a)\bar{\chi'(aA+1)})^r,$$
where $\chi'$ is a character modulo $c\ell^{\nu}$. Since $\chi'(aA+1) = 1$, the second factor vanishes. Hence, \eqref{REDMOD} holds, and we conclude that $f(A)g(n)=g(An+1) + O(\e^2)$ infinitely often in $n$. This contradiction implies that $q = p^l$ in the first place. \\
c) Assume that $\chi$ has order $m$. Fix, for the moment, an integer $A \equiv 1 (q)$. As in Lemma \ref{POSDENS}, we can choose $z$ with $z^{mk}=1$ and find infinitely many solutions $n$ to the equation $g(n) = \bar{z}g(An+1)$. It is clear that this choice of $z$ depends only on the residue class of $A$. Now suppose there exists a prime $b \neq p$ for which $f(b) = e(\alpha)$ and $\alpha \notin \mb{Q}$. In this case we necessarily have $|f(b)^k - \chi(b)| \gg 1$, and $f(b)^{mk}$ still has irrational argument. By Kronecker's theorem, we can choose $J$ such that $f(b)^{mkJ} = z + O(\e^2)$. Now, by Lemma \ref{SELECTPRIME} we can find a prime $c \equiv \bar{b} (q)$ such that $$f(c)^{mkJ} = g(c)^{mkJ} + O(\e^2) = 1 + O(\e^2).$$ Assuming that $\e$ is sufficiently small, and setting $b' := bc$, we see that $f(b')^{mkJ} = z + O(\e^2)$. Let $A = (b')^{mkJ}$, so that $A \equiv 1 (q)$. We thus have $f(A) = z + O(\e^2)$, whence we can find integers $n$ such that $|g(n)-\bar{f(A)}g(An+1)| \ll \e^2$, a contradiction when $\e$ is sufficiently small. \\
Suppose, instead, that the sequence $\{f(b)\}_b$, where $b$ is an integer for which $f(b)^k \neq \chi(b)$, consists of unimodular complex numbers whose arguments are all \emph{rational}, but such that the denominators of these arguments are unbounded. Write $f(b) = e(r_b/s_b)$, where $(r_b,s_b) = 1$, and suppose $s_b \gg mk\e^{-2}$. Without loss of generality, we may assume that $r_b = 1$, by replacing $b$ by $b^d$, where $d \equiv \bar{r_b} (s_b)$. Let $z$ be as in the irrational case, and write $z = e(u/v)$. We choose $J = \llf \frac{us_b}{mk}\rrf$. In this case, we have
$$\left|J\frac{r_b}{s_b} - \frac{u}{v} \right| \ll \frac{\e^2}{mk} \left|J - umk s_b\right| \leq \e^2,$$
so that $f(b)^{mkJ} = z + O(\e^2)$. Following the argument in the irrational case, we can thus choose $b' \equiv 1 (q)$ for which $f(b')^{mkJ} = z + O(\e^2)$, which is sufficient to get the required contradiction as in the previous case. \\
Hence, we may assume that $f$ has finite order, say $M$, away from multiplies of $p$. Since $f(p') = g(p')$ for most primes $p'$ and $g$ has order $mk$, it follows that $mk|M$.
\end{proof}
By c) of the previous lemma, $f(p') = g(p')$ except on a thin set of primes $S$.
We will need some information about the behaviour of arithmetic functions restricted to the integers generated by primes in $S$. We define $\pi_S(n)$ to be the unique integer $m \in \lla S\rra$ such that $m|n$ and $p |n/m \Rightarrow p \notin S$. 
We formally write $P_S$ to denote the product of all primes in $S$, and write $d|P_S$ to mean that $d \in \lla S\rra$.
We also define $\tau_S(n)$ be the count of divisors of $n$ all of whose prime factors lie in $S$, and $\Omega_S(n)$ to denote the count (with multiplicity) of prime factors of $n$ that belong to $S$.  These three functions are related by the expressions $\tau_S(\pi_S(n)) = \tau_S(n)$ and $\tau_S(n) \leq 2^{\Omega_S(n)}$. 
\begin{lem} \label{THINTRIV}
Let $S$ be a thin set. The following estimates hold: \\
i) $\displaystyle \sum_{n \leq x \atop n \in \lla S\rra} \frac{\tau_S(n)}{n} \ll 1$; \\
ii) $\displaystyle \sum_{n \leq x \atop n \in \lla S\rra} \frac{\tau_S(n)\log n}{n} = o(\log x).$
\end{lem}
\begin{proof}
i) It is clear that $\tau_S(n)$ is multiplicative. As such, we have
\begin{align}
\sum_{n \leq x} \frac{\tau_S(n)1_{n\in \lla S\rra}}{n} &= \sum_{k \geq 0} \frac{\tau_S(2^k)1_{2 \in S}}{2^k} \sum_{n \leq x/2^k} \frac{\tau_S(n)1_{n\in \lla S\rra}1_{(n,2) = 1}}{n} \nonumber\\
&\leq \sum_{k \geq 0} \frac{k+1}{2^k} \sum_{n \leq x/2^k} \frac{2^{\Omega_S(n)}1_{n\in \lla S\rra}1_{(n,2) = 1}}{n}. \label{INNEROUTER}
\end{align}
The inner sum is bounded above by
\begin{equation*}
\prod_{3 \leq p \leq x \atop p \in S} \left(1-\frac{2}{p}\right)^{-1} \ll \exp\left(2\sum_{p \leq x \atop p \in S} \frac{1}{p}\right) \ll 1,
\end{equation*}
where the last bound comes from the fact that $S$ is thin. As the outer sum in \eqref{INNEROUTER} converges, this implies the first estimate.\\
ii) Using the identity $\log = 1\ast \Lambda$ and the inequality $\tau_S(ab) \leq \tau_S(a)\tau_S(b)$, we get
\begin{equation*}
\sum_{n \leq x} \frac{\tau_S(n)}{n} \log n \leq \sum_{p^k \leq x \atop p \in S} \frac{(k+1)\log p}{p^k} \sum_{m \leq x/p^k} \frac{\tau_S(m)1_{m \in \lla S\rra}}{m} \ll 1+ \sum_{p \leq x\atop p \in S} \frac{\log p}{p},
\end{equation*}
where the second estimate follows from part i). Splitting the sum at height $\log x$, we get that
\begin{equation*}
\sum_{p \leq x \atop p \in S} \frac{\log p}{p} \leq \sum_{p \leq \log x} \frac{\log p}{p} + (\log x)\sum_{p > \log x \atop p \in S} \frac{1}{p} = o(\log x),
\end{equation*}
since the first term is $O(\log \log x)$ by the Prime number theorem, while in the second we have $\sum_{p > \log x} 1_{p \in S}/p = o(1)$, owing to the thinness of $S$. The proof is complete.
\end{proof}
For convenience, given $K \in \mb{N}$ we write $\mu_K := \{z \in \mb{T} : z^K = 1\}$.
\begin{lem} \label{DENSPROJ}
Let $S$ be a thin set. Let $a,b \in \lla S\rra$ with $2|ab$, and let $g: \mb{N} \ra \mu_{kl}$ be a completely multiplicative function such that $g^j$ is non-pretentious for all $1 \leq j \leq k-1$, and  $g^k = \tilde{\chi}$ a primitive character modulo $q$. Then for each $\zeta \in \mu_k$ the set $$\{n : \pi_S(n) = a, \pi_S(n+1) = b, g(n) = \zeta g(n+1)\}$$ has positive logarithmic density.
\end{lem}
\begin{proof}
The logarithmically averaged count of elements in the set in question is
\begin{align*}
&\sum_{n \leq x} \frac{1_{\pi_S(n) = a} 1_{\pi_S(n+1)=b} 1_{g(n) = \zeta g(n+1)}}{n} = \frac{1}{ak}\sum_{j(kl)} \zeta^{-j} \sum_{m \leq x/a} 1_{\pi_S(m) = 1} 1_{\pi_S(am+1) = b} \frac{(g(am)\bar{g(am+1)})^j}{m} \\ 
&= \frac{1}{ak}\sum_{j(kl)} \zeta^{-j} \sum_{d|P_S,e|P_S/\text{rad}(b) \atop (ad,eb) = 1}\frac{\mu(d)\mu(e)}{d} \sum_{m \leq x/ad}1_{m \equiv -\bar{ad} (eb)} \frac{(g(adm)\bar{g(adm+1)})^j}{m} \\
&= \frac{1}{abk} \sum_{j(kl)} \zeta^{-j} \sum_{d|P_S,e|P_S/\text{rad}(b) \atop (ad,eb) = 1, de \leq x/ab}\frac{\mu(d)\mu(e)}{ed} \sum_{n \leq x/adeb} \frac{(g(adebn-1)\bar{g(adebn)})^j}{n} + O(1).
\end{align*}
For each $j = rk + s$ with $1 \leq s \leq k-1$, the inner sum over $n$ is $o(\log x)$ by Theorem \ref{TAOBINTHM}. Hence, the contribution from these terms is 
\begin{align*}
o\left((\log x)\sum_{de\leq x/ab \atop d,e \in \lla S\rra} \frac{1}{de}\right) = o\left((\log x)\sum_{n \leq x \atop n \in \lla S\rra} \frac{\tau_S(n)}{n}\right) = o(\log x),
\end{align*}
by i) of Lemma \ref{THINTRIV}. For the sums with $s = 0, r \neq 0$ we can split the sum into residue classes modulo $q$ as before (noting that as $p\notin S$, $(eb,q) = 1$) to get
\begin{align*}
&\sum_{1 \leq r \leq l-1} \zeta^{-rk}\sum_{n\leq x/abde} \frac{(\chi(adebn-1)\bar{\chi(adebn)})^r}{n} \\
&= \sum_{1 \leq r \leq l-1} \zeta^{-rk}\sum_{c(q)} \chi^r(c-1)\chi^r(c) \sum_{n \leq x/abde} \frac{1_{n \equiv c\bar{adeb}(q)}}{n} \\
&= -\frac{\log(x/abde)}{\phi(q)}\sum_{0 \leq r \leq l-1} \zeta^{-rk} + \frac{1}{\phi(q)}\log(x/abde) + O(1) \\
&= \frac{1}{\phi(q)}\log(x/abde) + O(1).
\end{align*}
The remaining contribution, from $r =s =0$, is
\begin{align}
&\frac{1}{abk} \sum_{r(l)} \zeta^{-rk}\sum_{d|P_S,e|P_S/\text{rad}(b) \atop (ad,eb) = 1, de \leq x/ab} \frac{\mu(d)\mu(e)}{ed}\left(\log (x/adeb)+O(1)\right) \nonumber\\
&= \frac{\log x}{abk} \sum_{d|P_S,e|P_S/\text{rad}(b) \atop (ad,eb) = 1, de \leq x/ab} \frac{\mu(d)\mu(e)}{de} + O\left(\frac{1}{abk} \sum_{n \leq x/ab \atop n \in \lla S\rra} \frac{\tau_S(n) \log n}{n}\right). \label{LEFTTODOEST}
\end{align}
By ii) of Lemma \ref{THINTRIV}, the error term is $o(\log x)$. Adding together the above estimates, it remains to show that the inner sum in the main term of  \eqref{LEFTTODOEST} is bounded away from zero. By i) of Lemma \ref{THINTRIV}, we can replace the inner sum by
\begin{align*}
&\sum_{de \in \lla S\rra \atop (e,b) = 1, (ad,eb) = 1} \frac{\mu(d)\mu(e)}{de} + O\left(\sum_{n > x/ab \atop n \in \lla S\rra}\frac{\tau_S(n)}{n}\right) = \sum_{d \in \lla S\rra} \frac{\mu(d)}{d}\prod_{p \in S \atop p \nmid [b,ad]} \left(1-\frac{1}{p}\right) + o(1) \\
&= \prod_{p \in S \atop p \nmid ab}\left(1-\frac{1}{p}\right) \sum_{d \in \lla \mc{S} \rra} \frac{\mu(d)}{d}\prod_{p|d/(d,ab)} \left(1-\frac{1}{p}\right)^{-1} + o(1) \\
&= \prod_{p \in S \atop p \nmid ab} \left(1-\frac{2}{p}\right)\prod_{p|ab} \left(1-\frac{1}{p}\right) + o(1),
\end{align*}
and this last expression is strictly positive given that $2|ab$. This completes the proof.
\end{proof}
\begin{prop} \label{REDUCEDTHM}
Let $\chi$ be a Dirichlet character modulo $q = p^l$, where $p$ is a fixed odd prime. Suppose $f$ is a completely multiplicative function with $\mb{D}(f^k,\chi;x) \ll 1$, $f(n)^M = 1$ for all $n$ coprime to $p$ and $M \in \mb{N}$, and $f(n+1) \neq f(n)$ for all $n$ sufficiently large. Let $S := \{r \text{ prime} : f(r)^k \neq \chi(r)\}$, and suppose there are primes in $S$ not congruent to 1 modulo $q$. Then if $c,d \in \lla S\rra$ are such that $f(c),f(d)$ are primitive $M$th roots of unity and $c \equiv d (q)$ then $f(c) = f(d)$.
\end{prop}
\begin{proof}
By assumption $S$ is thin. Since one can find $p' \in S$ such that $p' \not \equiv 1 (q)$, $\lla S\rra$ contains integers in at least two residue classes modulo $q$. Let $a'$ and $b'$ be representatives of this set. Suppose otherwise that there are $c \equiv d(q)$ with $f(c),f(d)$ primitive roots of unity and $f(c) \neq f(d)$. Let $a',b' \in \lla S\rra$ be chosen such that $a' \not \equiv b' (q)$, and that $2|a'b'$ (if $2 \notin S$ then we can add it in without any harm to the remainder of the argument). If $f(a') \neq f(b')$, choose $r$ such that $f(a')\bar{f(b')} = (f(c)\bar{f(d)})^r$, and let $a := a'd^r$ and $b := b'c^r$. We clearly still have $a \not \equiv b(q)$, and $f(a) = f(b)$. We may assume that $a \not \equiv 1 (q)$.\\
We now observe that if $\pi_S(n) = a$ and $\pi_S(n+1) = b$ and $\chi(n) = \chi(a\bar{b})\chi(n+1)$ then
\begin{equation*}
f(n)^k-f(n+1)^k = f(a)^k\bar{\chi(a)}\chi(n) - f(b)^k\bar{\chi(b)}\chi(n+1) = 0,
\end{equation*}
in which case $f(n) = \zeta f(n+1)$, where $\zeta$ is a $k$th root of unity. 
Note furthermore that 
$$\zeta = f(n)\bar{f(n+1)} = g(n/a)\bar{g((n+1)/b)} = \bar{g(a)}g(b)g(n)\bar{g(n+1)}.$$
By the previous lemma, we may choose infinitely many integers $n$ such that $\pi_S(n) = a$, $\pi_S(n+1) = b$ and $g(n) = g(a)\bar{g(b)}g(n+1)$ (which automatically satisfies $\chi(n) = \chi(a\bar{b}) \chi(n+1)$ since $\zeta' \in \mu_k$), in which case $\zeta = 1$. We can therefore find infinitely many integers $n$ contradicting the assumption that $f(n) \neq f(n+1)$ for all large $n$. This contradiction completes the proof.
\end{proof}
\begin{proof}[Proof of Theorem \ref{DISC}]
By Lemma \ref{REDUCTIONS} we can reduce to $f$ satisfying the hypotheses of Proposition \ref{REDUCEDTHM}. Applying the latter completes the proof of Theorem \ref{DISC}.
\end{proof}
\begin{rem}
In some cases, Theorem \ref{DISC} actually implies that $S$ consists only of primes congruent to 1 modulo $q$. To see this, let $r$ be a prime and $k$ a positive integer such that is $kr$ is coprime to the product $\prod_{d|\phi(q)} (kd+1)$. Let $f$ be a completely multiplicative function for which there is an integer $n \in \lla S\rra$ such that $n\equiv a(q)$, $a \not \equiv 1(q)$, and $\text{ord}(f(n)) = rk\phi(q)$ (in the notation of the theorem). Let $t=\text{ord}_q(a)$. Clearly, $kt+1$ is coprime to $rk \phi(q)$ since $t|\phi(q)$. Thus, $n^{kt+1} \equiv n (q)$, and $f(n^{t+1})$ still has order $rk\phi(q)$. As such, Theorem \ref{DISC} implies that $f(n^{kt+1}) = f(n)$, a contradiction since $f(n)^{kt} \neq 1$ as $kt < rk\phi(q)$. Hence, in this case the elements of $\lla S \rra$ necessarily arise as products of primes congruent to 1 modulo $q$. 
\end{rem}
We conclude with a some examples showing the minimality of the conditions above.
\begin{ex}
i) Let $\chi$ be the Dirichlet character modulo 9 taking values $$\chi(1) = 1, \ \chi(2) = \zeta_3, \ \chi(4) = \zeta_3^2, \ \chi(5) = -\zeta_3^2, \ \chi(7) = -\zeta_3, \chi(8) = 1.$$ This clearly satisfies $|\chi(n)-\chi(n+1)| \geq 1$ for all $n$, and provides a first example. \\
ii) Next, let $k \geq 2$ be fixed and define $g(n) = e(\theta_n/k)$, where $\chi(n) = e(\theta_n)$, whenever $(n,3) = 1$. Then, for instance, $g(1) = e(1/k)$, $g(2) = e(1/3k)$ and so on. One still finds that $|g(n)-g(n+1)| \geq \e$, here with $\e = \sqrt{2(1-\cos(2\pi/3k))}$. This furnishes a second. \\
iii) Evidently, we can shift this last example by an archimedean character $n^{it}$ without affecting matters when $n$ is large. Thus, $g(n)n^{it}$ provides a third counterexample. \\
iv) Lastly, we show that the exceptional set $S$ can be infinite, though thin. We suppose $S$ is an arbitrary thin set consisting solely of primes $r \equiv 1(q)$. We choose a completely multiplicative function $f$ on primes by 
$$f(r)r^{-it} = \begin{cases} g(r) &: \ r \notin S \\ e(1/\ell k) &: r \in S. \end{cases}.$$
Here $\ell$ is a prime coprime to 6 that satisfies $(kd+1,6\ell k) =1$ for all $d|6$. 
If $n$ is sufficiently large, $\pi_S(n) = a$ and $\pi_S(n+1) = b$ and $f(n) = f(n+1)$, we would have
\begin{equation*}
e\left(\frac{\Omega(a)}{6\ell}\right) \chi(n)\bar{\chi(a)} = f(n)^k = f(n+1)^k = e\left(\frac{\Omega(b)}{6\ell}\right) \chi(n+1)\bar{\chi(b)}.
\end{equation*}
Since $a,b \in \lla S\rra$, we have $\chi(a) = \chi(b) = 1$. Now, if $\Omega(a) \equiv \Omega(b) (\ell)$ then we get a contradiction because $\chi(n) \neq \chi(n+1)$ for all $n$ large. On the other hand, if $\Omega(a) \not \equiv \Omega(b) (\ell)$ then we have
\begin{equation*}
e\left(\frac{\Omega(a)-\Omega(b)}{\ell}\right) = \chi(n+1)\bar{\chi(n)},
\end{equation*}
which is impossible since the term on the right side has order dividing 6, while the term on the right side has order $\ell$. 
\end{ex}
We now construct an example in which $f$ fixes entire residue classes. 
\begin{ex}
Let $\chi$ be a character modulo a prime $p$ of exact order $p-1$. Such a character will always satisfy $\chi(n) \neq \chi(n+1)$. Indeed, the map $k \mapsto \rho^k$ is a bijection on $\mb{F}_p^{\times}$ for any primitive root $\rho$ modulo $p$. Thus, $\chi$ will separate the residue classes generated by these powers, and is thus injective on $\mb{Z}/p\mb{Z}$. \\
Fix $\rho$ a primitive root such that $\chi(\rho) = e(1/(p-1))$. Now, for $k$ fixed, select $g$ in the same fashion as above (taking $k$th roots by dividing the argument by $k$). For $2 \leq m < p$ let $a_1,\ldots,a_m$ be distinct residue classes modulo $p$, and write $a_j \equiv \rho^{\nu_j} (p)$. Let $S$ be a union of thin sets $S_1,\ldots,S_m$ such that for every $r \in S_j$, $r \equiv a_j (p)$.  We now define $f(r) = g(r)$ for $r \notin S$, and for each $r \in S_j$ we let $f(r) = e\left(\frac{\nu_j}{(p-1)\ell k}\right)$, where $\ell$ is a prime distinct from $p$ chosen so that $(kd+1,(p-1)\ell k) = 1$ for all $d|(p-1)$, and such that $\ell \equiv 1 (p-1)$. \\
We now verify that $f(n) \neq f(n+1)$ for all $n$. Indeed, if there were $n$ such that $f(n) = f(n+1)$ then if $\pi_S(n) = a$ and $\pi_S(n+1) = b$ we get
\begin{align}
&e\left(\frac{1}{(p-1)\ell} \sum_j \nu_j\Omega_{S_j}(a)\right) \bar{\chi(a)} \chi(n) = f(a)^k\chi(n\bar{a}) = f(n)^k = f(n+1)^k \nonumber \\
&= f(b)^k \chi((n+1)\bar{b}) = e\left(\frac{1}{(p-1)\ell k}\sum_j \nu_j \Omega_{S_j}(b)\right) \bar{\chi(b)}\chi(n). \label{EQUALITY}
\end{align}
Let $\Delta(a,b):= \sum_j \nu_j\left(\Omega_{S_j}(a)-\Omega_{S_j}(b)\right)$. If $\Delta(a,b) \equiv 0 ((p-1)\ell)$ then $\chi(a) = \chi(b)$. For indeed, writing $c \equiv \rho^{\nu_c} (p)$ for $c \in \{a,b\}$, we have
\begin{equation*}
\chi(a)\bar{\chi(b)} = \chi(\rho)^{\nu_a-\nu_b} = e\left(\frac{1}{p-1} \sum_j \nu_j\left(\Omega_{S_j}(a)-\Omega_{S_j}(b)\right)\right) = 1.
\end{equation*}
In this case, $\chi(n) = \chi(n+1)$, which is impossible by assumption. \\
Similarly, if $\Delta(a,b) \not \equiv 0 ((p-1)\ell)$ but $\Delta(a,b) \equiv 0 (\ell)$ then, writing $\Delta(a,b) = \ell \delta$, we in fact have 
$$f(a)^k\bar{f(b)}^k = e\left(\frac{\delta}{p-1}\right) = e\left(\frac{\ell \delta}{p-1}\right) = \chi(a)\bar{\chi(b)},$$
the middle equality owing to the hypothesis $\ell \equiv 1 (p-1)$. It thus follows from \eqref{EQUALITY} that $\chi(n) = \chi(n+1)$, which again is not possible. \\
If, instead, $\Delta(a,b) \not \equiv 0 ((p-1)\ell)$ and $\Delta(a,b) \not \equiv 0 (\ell)$ then $f(a)\bar{f(b)}$ has order $\ell$, while the remaining terms are either zero or have order dividing $p-1$. Thus, $f(n) \neq f(n+1)$ in this case as well.
\end{ex}
\section{On Chudakov's Conjecture}\label{SECCHUD}
Suppose $f$ is a completely multiplicative function that satisfies the hypotheses of Theorem \ref{CHUDAKOV} in the case $\alpha = 0$, i.e., $f$ has finite range, vanishes on a finite, non-empty set of primes, and has bounded partial sums. \\
Observe first that the non-zero values of $f$ must be roots of unity. Indeed, it is clear that if, for some prime $p$, $f(p) = re(\theta)$ where $r > 0$ and $\theta \in \mb{R}$ then unless $r = 1$, we have $|h(p^k)| = r^k$ which yields infinitely many values. Moreover, in order for the set of values to be finite, one must also have $e((k-l)\theta) = 1$ for some pair of distinct positive integer $k,l$, whence that $\theta \in \frac{1}{k-l} \mb{Z} \subset \mb{Q}$.\\
Next, let $\mc{S}$ be the set of primes at which $f$ vanishes and let $P = P_{\mc{S}} := \prod_{p \in \mc{S}} p$. If $f$ is to be a character then it ought to have modulus $q$ where $\text{rad}(q) = P$. Note in particular that the previous paragraph implies that
\begin{equation} \label{ZEROTH}
\sum_{n \leq x} |f(n)|^2 = \sum_{n \leq x \atop (n,P) = 1} 1 = \frac{\phi(P)}{P}x + O(\tau(P)) = \sum_{n \leq x} |\chi(n)|^2 + O(1),
\end{equation}
whenever $\chi$ is chosen with modulus $q$ such that $\text{rad}(q)= P$. This is a consequence of the analysis below.
\begin{lem} \label{NOTS}
If $f$ satisfies the hypotheses of Theorem \ref{CHUDAKOV} then there is a Dirichlet character $\chi$ modulo $q$ such that $\mb{D}(f,\chi;\infty) \ll 1$. If $2 \notin \mc{S}$ then we may choose $\chi$ to be primitive; otherwise, $\chi$ can be chosen such that $2|q$.
\end{lem}
\begin{proof}
The same argument as in the proof of Lemma 4.3 in \cite{Klu} implies that $\mb{D}(f,\chi n^{it};\infty) \ll 1$ for some primitive character $\chi$ modulo $q'$ and some $t \in \mb{R}$.
If $2 \notin S$ then we may take $q = q'$. Otherwise, replacing $\chi$ by $\chi \chi_0^{(2)}$, where $\chi_0^{(2)}$ is the principal character modulo $2$ if necessary, we can take $q = 2q'$ (and in this case $\chi$ might not be primitive).\\
Since $f(n)$ is a root of unity wherever it is non-zero, there are positive integers $N$ and $m$ such that $(f1_{f \neq 0})^N = 1$ and $\chi^m = \chi_0$, where $\chi_0$ is the principal character having the same modulus as $\chi$ does. Suppose now that $t \neq 0$, and fix $l \in \mb{N}$, chosen so that $lmNt > 1.$ By the triangle inequality, 
\begin{equation*}
\mb{D}(1,n^{ilmNt};x) \leq \mb{D}(f^{lmN},\chi^{lmN}n^{ilmNt};x) \leq lmN\mb{D}(f,\chi n^{it};\infty) \ll 1.
\end{equation*}
In other terms, we have
\begin{equation*}
O(1) = \log_2 x - \text{Re}\left(\sum_{p \leq x} p^{-1-ilmNt}\right) = \log_2 x- \log|\zeta(1+ilmNt)| + O(1),
\end{equation*}
as $x \ra \infty$. But this is false, since $\log|\zeta(1+imNt)| \ll_t 1$. Hence, $t = 0$ after all.
\end{proof}
We now define the completely multiplicative function $F(n)$ implicitly via $f(n) = \chi(n)F(n)$ if $(n,q) = 1$, and $F(p) = 1$ for each $p|q$. Note that if $q$ is odd then $F(2) \neq 0$, while if $q$ is even then we have $F(2) = 1$ by necessity. \\
For $e \in \mb{N}_0$, define
\begin{equation*}
G(e) := \prod_{p^k ||e \atop k \geq 0} \left(|\mu \ast F(p^k)|^2 + 2\text{Re}\left(\sum_{i > k} \frac{\mu \ast F(p^i)\bar{\mu\ast F(p^k)}}{p^{i-k}}\right)\right).
\end{equation*}
Consider first $k = 0$. Then $\mu \ast F(p^k) = 1$ and as $\mu \ast F(p^i) = F(p)^{i-1}(F(p)-1)$ for all $i > 0$, we get that each factor with $p^k || e$ and $k = 0$ has the form $1 + 2\text{Re}\left(\frac{F(p)-1}{p-F(p)}\right)$. When $k \geq 1$, a similar computation shows that the factor is then $|F(p)-1|^2 \left(|F(p)|^{k-1}+ 2\text{Re}\left(\frac{F(p)}{p-F(p)}\right)\right)$. It follows that
\begin{equation*}
G(e) = \prod_{p \nmid e} \left(1+2\text{Re}\left(\frac{F(p)-1}{p-F(p)}\right)\right) \prod_{p^k||e} |F(p)-1|^2\left(|F(p)|^{k-1}+2\text{Re}\left(\frac{F(p)}{p-F(p)}\right)\right).
\end{equation*}
Note that $G(1) \neq 0$. Writing $\tilde{G}(e) := G(e)/G(1)$, we produce the multiplicative function 
\begin{align}
\tilde{G}(e) &= \left(\prod_{p^k||e \atop F(p) = 0} \left(1-\frac{2}{p}\right)^{-1}|F(p)|^{k-1}\right)\prod_{p|e \atop F(p) \neq 0}|F(p)-1|^2\text{Re}\left(\frac{p+F(p)}{p-F(p)}\right)\text{Re}\left(\frac{p-2 + F(p)}{p-F(p)}\right)^{-1} \nonumber\\
&= \left(\prod_{p^k||e \atop F(p) = 0} \left(1-\frac{2}{p}\right)^{-1}|F(p)|^{k-1}\right)\prod_{p|e \atop F(p) \neq 0}|F(p)-1|^2\frac{\text{Re}\left((p+F(p))(p-\bar{F(p)})\right)}{\text{Re}\left((p-2 + F(p))(p-\bar{F(p)})\right)} \nonumber\\
&= \left(\prod_{p^k||e \atop F(p) = 0} \left(1-\frac{2}{p}\right)^{-1}|F(p)|^{k-1}\right)\prod_{p|e \atop F(p) \neq 0}|F(p)-1|^2\frac{p^2-1}{(p-1)^2 - 2(1-\text{Re}(F(p)))}. \label{TILDEG}
\end{align}
In fact, at primes $p$ for which $F(p) \neq 0$, we have moreover that $\tilde{G}(p^k) = \tilde{G}(p)$, i.e., that $\tilde{G}$ is \emph{strongly multiplicative}. We will use this fact repeatedly in the sequel. \\
The relevance of this function comes from
\begin{equation} \label{GFFORMULA}
G_f(d) := \lim_{x \ra \infty} x^{-1} \sum_{n \leq x} f(n)\bar{f(n+d)} = \frac{1}{q}\sum_{R|d\atop \text{rad}(R) | q}\frac{|f(R)|^2}{R} \sum_{a(q)} \chi(a)\bar{\chi(a+d/R)} \sum_{e|d/R} \frac{G(e)}{e}, 
\end{equation}
a formula that is implicit in the proof of Theorem 1.5 in \cite{Klu}.
\begin{lem}\label{QPS}
The following holds:
\begin{enumerate}[(i)]
\item whenever $(d,q) > 1$, $G(d) = 0$;\\

\item if $(d,2q) = 1$ then $\tilde{G}(d) \geq 0$; \\
\item the series $\sum_{d\geq 1 \atop (d,q) = 1} G(d)/d$ converges and is positive; and \\
\item if $\tilde{G}(2) < 0$ then $|\tilde{G}(2)| > 1$.
\end{enumerate}
\end{lem}
\begin{proof}
As mentioned, the function $\tilde{G}$ is strongly multiplicative at primes at which $F(p) \neq 0$. Since $F(p) = 1$ by construction whenever $p |q$, \eqref{TILDEG} implies that $\tilde{G}(p) = 0$ for $p|q$, and hence for any $(d,q) > 1$ we have $G(d) = 0$, proving i). \\
Statement ii) also follows from \eqref{TILDEG}, as $(p-1)^2 \geq 4 \geq 2(1-\text{Re}(F(p)))$ for all $p \geq 3$. \\
Observe now that when $d = 0$ in \eqref{GFFORMULA} ,
\begin{align*}
\frac{\phi(P)}{P} &= \lim_{x \ra \infty} \left(x^{-1}\sum_{n \leq x} |f(n)|^2\right) = \frac{1}{q}\left(\sum_{a(q)} |\chi(a)|^2\right)\sum_{\text{rad}(R)|q} \frac{|f(R)|^2}{R}\sum_{d \geq 1 \atop (d,q) = 1} \frac{G(d)}{d} \\
&= \frac{\phi(q)}{q}\prod_{p|q} \left(1-\frac{|f(p)|^2}{p}\right)^{-1} \sum_{d \geq 1 \atop (d,q) = 1} \frac{G(d)}{d} = \frac{\phi(q)}{q}\prod_{p|q \atop p \nmid P} \left(1-\frac{1}{p}\right)^{-1} \sum_{d \geq 1 \atop (d,q) = 1} \frac{G(d)}{d},
\end{align*}
which implies iii). \\
Finally, iv) follows from iii). Indeed, we know that $\sum_{(d,q) = 1} G(d)/d$ is positive, and moreover as $F(2) \neq 0$, $\tilde{G}$ is strongly multiplicative at the prime 2, and we have
\begin{equation*}
\sum_{(d,q) = 1} \frac{G(d)}{d} = G(1)\left(\sum_{(d,2q) = 1} \frac{\tilde{G}(d)}{d} + \sum_{k \geq 1} \frac{\tilde{G}(2^k)}{2^k}\sum_{(d,2q) = 1} \frac{\tilde{G}(d)}{d}\right) = G(1)\left(1 + \tilde{G}(2)\right)\sum_{(d,2q) = 1} \frac{\tilde{G}(d)}{d}.
\end{equation*}
Since $\tilde{G}(d) \geq 0$ for all odd $d$, we must have $G(1)(1+\tilde{G}(2)) > 0$. Since the signs of $G(1)$ and $\tilde{G}(2)$ are the same, it follows that $1+\tilde{G}(2) < 0$. This implies iv).
\end{proof}
Our first main goal is the following. Here and in the sequel, the terms $O(1)$ refer to boundedness as the parameter $H$ tends to infinity.
\begin{prop} \label{POSCOEFFS}
For any $H$ sufficiently large we have
\begin{equation} \label{FULLSUM}
\sum_{d \geq 1 \atop (d,2^{\tau}q) = 1} \tilde{G}(d)\sum_{\text{rad}(R)|q} |f(R)|^2\sum_{g|\text{rad}(q)/2^{\kappa}} \mu(g)g^{-2} \left\|\frac{Hg}{dR}\right\| = O(1),
\end{equation}
where $\tau = 1$ if $\tilde{G}(2) < 0$ and $\tau = 0$ otherwise, and $\kappa = 1$ if $q$ is even and $\kappa = 0$ if $q$ is odd.
\end{prop}
The interest in this expression stems from the fact that by ii) of Lemma \ref{QPS}, $\tilde{G}(d) \geq 0$ for all $(d,2q) = 1$. We will eventually show that the inner sum is always non-negative as well. This will imply that only finitely many values of $\tilde{G}(d)$ are non-zero, a crucial element of the proof of Theorem \ref{CHUDAKOV}. \\
We will first deduce the following similar, but weaker, estimate. 
\begin{lem} \label{NICE}
For any $H$ sufficiently large we have
\begin{equation*}
\sum_{d \geq 1 \atop (d,q) = 1} G(d)\sum_{\text{rad}(R)|q} |f(R)|^2\sum_{g|\text{rad}(q)} \mu(g)g^{-2} \left\|\frac{Hg}{dR}\right\| = O(1).
\end{equation*}
\end{lem}
\begin{proof}
Fix $H$ large. Since the partial sums of $f$ are all uniformly bounded, 
\begin{align}
O(1) &= \lim_{x \ra \infty} x^{-1}\sum_{n \leq x} \left|\sum_{n < m \leq n+H}f(m)\right|^2 = \sum_{|h| \leq H} (H-|h|)G_f(|h|) \nonumber\\
&= \sum_{d \geq 1 \atop (d,q) = 1} \frac{G(d)}{d}\sum_{\text{rad}(R)|q} \frac{|f(R)|^2}{R}\sum_{|h| \leq H \atop R| h, d|h/R} (H-|h|)S_{\chi}(|h|/R) \label{DSUM},
\end{align}
where we have dropped the terms $(d,q) > 1$ by i) of Lemma \ref{QPS}, and used \eqref{GFFORMULA}. Here, we have set
\begin{equation*}
S_{\chi}(h) := \sum_{a(q)} \chi(a)\bar{\chi(a+h)}.
\end{equation*}
We split the sum over $h$ up into the term $h = 0$ and the remainder. Thus, for each fixed $d$ coprime to $q$ and fixed $R$ with $\text{rad}(R)|q$,
\begin{align} \label{HSUM}
\sum_{|h| \leq H \atop R|h, d|h/R} (H-|h|)S_{\chi}(|h|/R) &= -H\frac{\phi(q)}{q} + \frac{2}{q}\left(\sum_{0 \leq  h \leq H/Rd}(H-hdR)\sum_{a(q)} \chi(a)\bar{\chi(a+hd)}\right).
\end{align}
Let us assume for the moment that $q$ is odd, and thus $\chi$ is primitive according to Lemma \ref{NOTS}. By the Chinese Remainder Theorem, we have
\begin{equation} \label{CRT}
\sum_{a(q)} \chi(a)\bar{\chi(a+hd)} = \prod_{p^k||q} \sum_{a(p^k)} \chi_{p^k}(a)\bar{\chi_{p^k}(a+hd)},
\end{equation}
where $\chi_{p^k}$ is the primitive character induced by $\chi$ on $(\mb{Z}/p^k\mb{Z})^{\ast}$. By primitivity, if $p^k||q$ and $\nu_p(hd) = \nu_p(h) = l$ then 
\begin{equation*}
\sum_{a(p^k)}\chi_{p^k}(a)\bar{\chi_{p^k}(a+hd)} = \begin{cases} 0 &\text{ if $l \leq k-2$} \\ -p^{k-1} &\text{ if $l = k-1$} \\ \phi(p^k) &\text{ otherwise}. \end{cases}
\end{equation*}
As such, we have
\begin{align*}
\sum_{a(q)} \chi(a)\bar{\chi(a+hd)} &= \prod_{p^k||q} \left(\phi(p^k)1_{p^k|h} - p^{k-1}1_{p^{k-1}||h}\right) = \prod_{p^k||q} \left(p^k1_{p^k|h}-p^{k-1}1_{p^{k-1}|h}\right) \\
&= q\prod_{p^k||q} \left(1_{p^k|h}-\frac{1}{p}1_{p^{k-1}|h}\right).
\end{align*}
Write $q = \prod_{1 \leq j \leq m} p_j^{\alpha_j}$. Given a subset $S$ of $\{1,\ldots,m\}$, write $q_S := \prod_{j \in S} p_j^{\alpha_j}$ and $q_S^{\ast} := q_S/\text{rad}(q_S)$ (note that $q = q_{\{1,\ldots,m\}}$, so $q^{\ast}$ is well-defined). Expanding the product above and summing over $h$, we have
\begin{align*}
\frac{1}{q}\left(\sum_{0 \leq h \leq H/dR} (H-hdR)\sum_{a(q)} \chi(a)\bar{\chi(a+hd)}\right) &= \sum_{S \subseteq \{1,\ldots,m\}} \frac{(-1)^{|S|}}{\text{rad}(q_S)}\sum_{0 \leq h \leq H/dR \atop q_S^{\ast}q_{S^c} | h} (H-hdR) \\
&= \sum_{S \subseteq \{1,\ldots,m\}} \frac{(-1)^{|S|}}{\text{rad}(q_S)}\sum_{0 \leq h \leq H/(q_S^{\ast}q_{S^c}dR) } (H-hdRq_S^{\ast}q_{S^c}).
\end{align*}
For a real number $t$, put $\Delta(t) := \{t\}-\{t\}^2$, where $\{t\}$ denotes the fractional part of $t$. Fixing a subset $S$ and evaluating the inner sum here yields
\begin{align*}
&H\left(1+\llf \frac{H}{q_S^{\ast}q_{S^c}dR}\rrf\right) - \frac{dRq_S^{\ast}q_{S^c}}{2}\left(\llf \frac{H}{dRq_S^{\ast}q_{S^c}}\rrf^2 + \llf \frac{H}{dRq_S^{\ast}q_{S^c}}\rrf\right) \\ 
&= H+\frac{H^2}{dRq_S^{\ast}q_{S^c}} - \left\{\frac{H}{dRq_S^{\ast}q_{S^c}}\right\} \\
&-\frac{dRq_S^{\ast}q_{S^c}}{2}\left(\frac{H^2}{(dRq_S^{\ast}q_{S^c})^2} - 2\frac{H}{dRq_S^{\ast}q_{S^c}}\left\{\frac{H}{dRq_S^{\ast}q_{S^c}}\right\} + \left\{\frac{H}{dRq_S^{\ast}q_{S^c}}\right\}^2 + \frac{H}{dRq_S^{\ast}q_{S^c}}-\left\{\frac{H}{dRq_S^{\ast}q_{S^c}}\right\}\right) \\
&= \frac{1}{2}\left(H + \frac{H^2}{dRq_S^{\ast}q_{S^c}} + dRq_S^{\ast}q_{S^c}\Delta\left(\frac{H}{dRq_S^{\ast}q_{S^c}}\right)\right).
\end{align*}
Note that $q_S^{\ast} q_{S^c}\text{rad}(q_S) = q$, and
\begin{equation*}
\sum_{S \subseteq \{1,\ldots,m\}} \frac{(-1)^{|S|}}{\text{rad}(q_S)} = \prod_{p|q}\left(1-\frac{1}{p}\right) = \frac{\phi(q)}{q},
\end{equation*}
so that upon summing over $S$, we get
\begin{align*}
&H\frac{\phi(q)}{2q} + \frac{H^2}{2qdR}\sum_{S \subseteq \{1,\ldots,m\}} (-1)^{|S|} + \frac{qdR}{2}\sum_{S \subseteq \{1,\ldots,m\}} \frac{(-1)^{|S|}}{\text{rad}(q_S)^{2}} \Delta\left(\frac{H\text{rad}(q_S)}{qdR}\right) \\
&= H\frac{\phi(q)}{2q} + \frac{qdR}{2}\sum_{S \subseteq \{1,\ldots,m\}} \frac{(-1)^{|S|}}{\text{rad}(q_S)^{2}}\Delta\left(\frac{H\text{rad}(q_S)}{qdR}\right).
\end{align*}
Inserting this expression back into \eqref{HSUM} and writing the sum over sets $S$ as a divisor sum over $g|\text{rad}(q)$ with $g = \text{rad}(q_S)$,
\begin{align*}
\sum_{|h| \leq H \atop R|h, d|h/R} (H-|h|)S_{\chi}(|h|/R) &= qdR\sum_{S \subseteq \{1,\ldots,m\}} \frac{(-1)^{|S|}}{\text{rad}(q_S)^{2}} \Delta\left(\frac{H\text{rad}(q_S)}{qdR}\right) \\
&=qdR\sum_{g|\text{rad}(q)} \mu(g)g^{-2}\Delta\left(\frac{Hg}{qdR}\right).
\end{align*}
Since $H$ was arbitrary, we can replace $H$ by $Hq$ and, upon inserting this last expression into the main term of \eqref{DSUM}, we get that
\begin{equation*}
S(H) := \sum_{d \geq 1 \atop (d,q) = 1} G(d)\sum_{\text{rad}(R)|q} |f(R)|^2\sum_{g|\text{rad}(q)} \mu(g)g^{-2} \Delta(Hg/dR) = O(1).
\end{equation*}
A short calculation shows that $$4\Delta(t) - \Delta(2t) = 2\left\|t\right\|$$ for each $t \in \mb{R},$
and therefore
\begin{align*}
O(1) &= 4S(H)-S(2H) = \sum_{d \geq 1 \atop (d,q) = 1} G(d) \sum_{\text{rad}(R)|q} |f(R)|^2\sum_{g|\text{rad}(q)} \mu(g)g^{-2} \left(4\Delta(Hg/dR)-\Delta(2Hg/dR)\right) \\
&= 2\sum_{d \geq 1 \atop (d,q) = 1} G(d) \sum_{\text{rad}(R)|q} |f(R)|^2\sum_{g|\text{rad}(q)} \mu(g)g^{-2}\left\|\frac{Hg}{dR}\right\|,
\end{align*}
which completes the proof in the case that $q$ is odd. \\
When $q$ is even and $\chi$ is primitive, the proof follows as before. Suppose then that $\chi = \chi_0^{(2^{\nu})}\chi'$, where $\chi'$ is primitive modulo $q'$ and $q'$ is odd, and $\chi_0^{(2^{\nu})}$ is the principal character modulo $2^{\nu}$ for some positive integer $\nu$. We note that $S_{\chi}(h) = 0$ whenever $h$ is odd. As such, it suffices to restrict the sum over $h$ in \eqref{DSUM} to additionally satisfy the condition $2|h$. The factorization \eqref{CRT} now contains the trivial character sum over $\mb{Z}/2^{\nu}\mb{Z}$ which is $\phi(2^{\nu})$, and the remaining character factors being as above. The proof then proceeds (with $H$ replaced by $H/2$) precisely as above. 
\end{proof}
\begin{proof}[Proof of Proposition \ref{POSCOEFFS}]
Assume first that $q$ is odd. The result is immediate from Lemma \ref{NICE} in the case that $\tilde{G}(2) \geq 0$. Thus, we shall assume that $\tilde{G}(2) < 0$. For sufficiently large $H$, define 
\begin{equation*}
\mc{M}(H) := \sum_{(d,2q) = 1} G(d)\sum_{\text{rad}(R)|q} |f(R)|^2\sum_{g|\text{rad}(q)} \mu(g)g^{-2}\left\|\frac{Hg}{dR}\right\|. 
\end{equation*}
Then Lemma \ref{NICE} and the strong multiplicativity of $\tilde{G}$ at $2$ together imply that
\begin{align}
O(1) &=  \sum_{(d,2q) = 1} \tilde{G}(d)\sum_{\text{rad}(R)|q} |f(R)|^2\sum_{g|\text{rad}(q)} \mu(g)g^{-2}\left\|\frac{Hg}{dR}\right\| \nonumber \\
&+ \tilde{G}(2)\sum_{(d,2q) = 1} \tilde{G}(d)\sum_{\text{rad}(R)|q} |f(R)|^2\sum_{g|\text{rad}(q)} \mu(g)g^{-2}\sum_{k \geq 1} \left\|\frac{Hg}{2^kdR}\right\| \nonumber\\
&= \mc{M}(H) + \tilde{G}(2) \sum_{k \geq 1} \mc{M}(H/2^k).\label{RECURM}
\end{align}
Applying \eqref{RECURM} with both $H$ and $H/2$ and subtracting the two, we get
\begin{equation*}
O(1)  = \mc{M}(H) + (\tilde{G}(2)-1) \mc{M}(H/2) = \mc{M}(H) - (1+|\tilde{G}(2)|)\mc{M}(H/2).
\end{equation*}
Let $\eta := 1+|\tilde{G}(2)| > 2$. By iv) of Lemma \ref{QPS}, $\eta > 2$. Let now $$T := \limsup_{H \ra \infty} |\mc{M}(H)-\eta^{-1}\mc{M}(2H)|.$$ Then for any $K \in \mb{N}$ we have
\begin{equation*}
|\mc{M}(H) - \eta^{-K} \mc{M}(2^KH)| \leq \sum_{0 \leq k \leq K-1} \eta^{-k} |\mc{M}(2^kH)-\eta^{-1}\mc{M}(2^{k+1}H)| \leq T\sum_{0 \leq k \leq K-1} 2^{-k} \leq 2T.
\end{equation*}
Invoking iii) of Lemma \ref{QPS}, we note that
\begin{equation*}
\mc{M}(2^KH) \ll \sum_{dR \leq 2^KH \atop (d,2q) = 1, \text{rad}(R)|q} 1 + 2^KH\sum_{\text{rad}(R)|q} \frac{|f(R)|^2}{R}\sum_{dR > 2^KH \atop (d,2q) = 1, \text{rad}(R)|q} \frac{\tilde{G}(d)}{d} \ll 2^KH.
\end{equation*}
As $\eta > 2$, it follows that $\lim_{K \ra \infty} \eta^{-K}\mc{M}(2^KH) = 0$ for any fixed $H$. Taking $K \ra \infty$, we conclude that $\mc{M}(H) = O(1)$, which is equivalent to the claim when $q$ is odd. \\
We assume now that $q$ is even. Define 
\begin{equation*}
\Sigma(H) := \sum_{d \geq 1 \atop (d,2^{\tau}q) = 1} \tilde{G}(d)\sum_{\text{rad}(R)|q} |f(R)|^2\sum_{g|\text{rad}(q)/2}\mu(g)g^{-2}\left\|Hg/dR\right\|.
\end{equation*}
Since $$\sum_{g|\text{rad}(q)} \mu(g)g^{-2}\left\|Hg/dR\right\| = \sum_{g|\text{rad}(q)/2} \mu(g)g^{-2}\left(\left\|Hg/dR\right\| - \frac{1}{4}\left\|2Hg/dR\right\|\right),$$ it follows from the above that $\Sigma(H)-\frac{1}{4}\Sigma(2H) = O(1)$ for all $H$. We will argue as in the previous paragraph. Define
\begin{equation*}
C := \limsup_{H \ra \infty}|\Sigma(H)-\frac{1}{4}\Sigma(2H)|.
\end{equation*}
Then, for any $K$, we have
\begin{equation*}
|\Sigma(H) - 2^{-2K}\Sigma(2^KH)| \leq \sum_{k \leq K-1} 2^{-2k} \left|\Sigma(2^kH)-\frac{1}{4}\Sigma(2^{k+1}H)\right| \leq C\sum_{k \geq 0} 2^{-2k} \leq 2C.
\end{equation*}
As before, we have the estimate
\begin{equation*}
|\Sigma(2^kH)| \ll 2^kH + 2^kH\sum_{dR > H2^k \atop (d,q) = 1} \frac{G(d)}{d} \ll 2^kH,
\end{equation*}
it follows that $2^{-2K} \Sigma(2^KH) \ra 0$ as $K \ra \infty$. Hence, we have $\Sigma(H) = O(1)$, i.e.,
\begin{equation*}
\sum_{d \geq 1 \atop (d,2^{\tau}q) = 1} \tilde{G}(d)\sum_{\text{rad}(R)|q} |f(R)|^2\sum_{g|\text{rad}(q)/2}\mu(g)g^{-2}\left\|Hg/dR\right\| = O(1),
\end{equation*}
as required.
\end{proof}
\begin{lem}\label{POSDIST}
For all $t > 0$ we have
\begin{equation*}
\sum_{g|\text{rad}(q)/2^{\kappa}} \mu(g)g^{-2}\left\|gt\right\| \geq 0,
\end{equation*}
where $\kappa = 1$ if $q$ is even, and $\kappa = 0$ otherwise.
\end{lem}
\begin{proof}
Fix $d$ for the time being. Recall the (uniformly convergent) Fourier expansion
\begin{equation*}
\|t\| = \frac{1}{4} - \frac{1}{2\pi^2}\sum_{k \neq 0} \frac{e(kt)}{k^2}(1-(-1)^k) = \frac{1}{4}-\frac{1}{\pi^2}\sum_{k \neq 0 \atop k \text{ odd}} \frac{e(kt)}{k^2}.
\end{equation*}
Assume $q$ is odd for the moment. We thus get
\begin{align*}
\sum_{g|\text{rad}(q)} \mu(g)g^{-2}\left\|gt\right\| &= \frac{1}{4}\prod_{p|q}\left(1-\frac{1}{p^2}\right) - \frac{1}{\pi^2}\sum_{g|\text{rad}(q)} \mu(g)g^{-2}\sum_{k \neq 0 \atop k \text{ odd}} \frac{e(kgt)}{k^2} \\
&= \frac{1}{4}\prod_{p|q}\left(1-\frac{1}{p^2}\right) - \frac{1}{\pi^2}\sum_{k \neq 0\atop k \text{ odd}} \frac{e(kt)}{k^2}\left(\sum_{g|\text{rad}(q)} \mu(g)1_{g|k}\right).
\end{align*}
Since each divisor $g$ of $\text{rad}(q)$ is necessarily squarefree, $1_{g|k} = \prod_{p|g}1_{p|k}$ and hence the above is
\begin{align*}
&\frac{1}{4}\prod_{p|q}\left(1-\frac{1}{p^2}\right) - \frac{1}{\pi^2}\sum_{k \neq 0\atop k \text{ odd}} \frac{e(kt)}{k^2}\prod_{p|q}(1-1_{p|k}) \\
&= \frac{1}{4}\prod_{p|q}\left(1-\frac{1}{p^2}\right)\left(1 - \frac{4}{\pi^2}\prod_{p|q}\left(1-\frac{1}{p^2}\right)^{-1}\sum_{k \neq 0 \atop (k,2q) = 1} \frac{e(kt)}{k^2}\right).
\end{align*}
It is now clear that the bracketed expression is real and non-negative. The former is true by the symmetry of the Fourier series, and the latter is true because
\begin{equation*}
\prod_{p|q} \left(1-\frac{1}{p^2}\right)^{-1}\sum_{k \neq 0 \atop (k,2q) = 1} \frac{e(kt)}{k^2} \leq 2\prod_{p|q}\left(1-\frac{1}{p^2}\right)^{-1}\prod_{p|2q} \left(1-\frac{1}{p^2}\right) \zeta(2) = \frac{3\pi^2}{12} = \frac{\pi^2}{4}.
\end{equation*}
If now $q$ is even then $\text{rad}(q)/2$ is odd. Proceeding as in the proof for odd $q$, we have that $\sum_{g|\text{rad}(q)/2} \mu(q)g^{-2}\left\|gt\right\| \geq 0$ for all $t > 0$ in this case. 
\end{proof}
\begin{proof}[Proof of Theorem \ref{CHUDAKOV}]
Fix a large positive real number $M'$. By Lemma \ref{POSDIST}, applied with $t = H/dR$ for each $(d,2q) = 1$ and $\text{rad}(R)|q$, each term in the outer sum in \eqref{FULLSUM} is non-negative. It follows then (upon dropping all $R$ except $R = 1$) that
\begin{equation*}
\sum_{d \leq M' \atop (d,2q) = 1} \tilde{G}(d)\sum_{g|\text{rad}(q)/2^{\kappa}} \mu(g)g^{-2}\left\|\frac{Hg}{d}\right\| = O(1),
\end{equation*}
where $\kappa$ is defined as in Proposition \ref{POSCOEFFS}. \\
Let $\{d_j\}_j$ be the set of $d$ coprime to $2q$ such that $\tilde{G}(d_j) \neq 0$ for all $j$. We shall assume for the sake of contradiction that the sequence $\{d_j\}_j$ is infinite. Choose $M$ such that $d_M \leq M' < d_{M+1}$, and put $H := \frac{1}{2}[d_1,\ldots,d_M]$. Then $\left\|Hg/d_j\right\| = 1/2$ for each $j \leq M$, since $Hg$ is a rational with odd numerator, and hence
\begin{align}
&O(1) = \sum_{d \leq M' \atop (d,2q) = 1} \tilde{G}(d)\sum_{g|\text{rad}(q)/2^{\kappa}} \mu(g)g^{-2}\left\|\frac{Hg}{d}\right\| \nonumber\\
&= \frac{1}{2}\sum_{j \leq M} \tilde{G}(d_j)\sum_{g|\text{rad}(q)/2^{\kappa}} \mu(g)g^{-2}= \frac{1}{2}\prod_{p|q \atop p \geq 3}\left(1-\frac{1}{p^2}\right)\sum_{j \leq M} \tilde{G}(d_j). \label{FIRSTREDUCTION}
\end{align}
By construction, $F$ takes only finitely many values since both $f$ and $\chi$ do. We may thus let $k$ be the minimal integer such that whenever $F(n) \neq 0$, $F(n)^k=1$. This means that if $F(p) \notin \{0,1\}$ then $1-\text{Re}(F(p)) \geq 1-\cos(2\pi/k)$. Note from \eqref{TILDEG} and the fact that $F$ is only zero at finitely many primes, that $$\tilde{G}(d) \gg (2(1-\cos(2\pi/k)))^{\omega(d)}$$ whenever $\tilde{G}(d) \neq 0$.  We consider two cases. \\
First, if $d_j$ is prime infinitely often then we may simply bound the sum in \eqref{FIRSTREDUCTION} from below by the contribution of these prime $d_j$, i.e.,
\begin{equation*}
N_M \ll_k \sum_{j \leq M \atop d_j \text{ prime}} \tilde{G}(d_j) \ll 1 ,
\end{equation*}
where $N_M$ is the number of prime $d_j$ with $j \leq M$. Since, by assumption, $N_M \ra \infty$, this is a contradiction. \\
Thus, we may assume that the $d_j$ are not prime infinitely often. Hence, let $N \geq 1$ be a bound for the number of prime values of $d_j$. Then each $d_j$ is composed of at most $N$ distinct prime factors, and hence $$\tilde{G}(d) \gg (2(1-\cos(2\pi/k)))^N \gg_k 1.$$ Hence, we again have
\begin{equation*}
M \ll_k \sum_{j \leq M} \tilde{G}(d_j) \ll 1,
\end{equation*}
and as $M \ra \infty$, we again get a contradiction. \\
Thus, there are only finitely many $d_j$ for which $G(d_j) \neq 0$. It is then clear that $G(d) = 0$ for all $d$ coprime to $2P$, where we recall that $P$ is the product of the primes in the zero set $\mc{S}$ of $f$. Indeed, if there is at least one prime $p$ dividing $2P$ such that $G(p) \neq 0$, then $G(p^k) = G(p) \neq 0$ for all $k \in \mb{N}$. Since, therefore, $G(p) = 0$ for all $p\nmid 2P$, it must be the case that $F(p) = 1$ whenever $p \nmid 2P$. This implies that $f(p) = \chi(p)$ for all $p \nmid [2P,q]$, and thus $f(n) = \chi(n)$ for all $(n,[2P,q]) = 1$. It thus remains to check that $f(n) = \chi(n)$ for all primes $p$ dividing $[2P,q]$.\\
We now check that $f$ is zero at primes dividing $q$. Applying Proposition \ref{POSCOEFFS} again, this time dropping all $d$ except $d = 1$, we see that
\begin{equation} \label{ONLYR}
\sum_{R \leq M' \atop \text{rad}(R)|q, R \text{ odd}} |f(R)|^2\sum_{g|\text{rad}(q)/2^{\kappa}} \mu(g)g^{-2} \left\|\frac{Hg}{R}\right\| = O(1).
\end{equation}
Proceeding as before, we choose $H$ such that $2H$ is the least common multiple of all odd $R \leq M'$ with $\text{rad}(R)|q/2^{\kappa}$ and at which $f(R) \neq 0$. Once again, this implies that $\left\|{Hg}/{R}\right\| = 1/2$ for each odd $R \leq M'$ at which $f(R) \neq 0$. The estimate \eqref{ONLYR} now shows that $f(R) \neq 0$ only finitely often. But if $f(R) \neq 0$ for some $R > 1$ then there is a prime $p$ dividing $q$ for which $f(p) \neq 0$, and hence by complete multiplicativity the same is true when $R = p^k$ for all $k \geq 1$, a contradiction. Hence, $f(R) = 0$ for all $R > 1$; in particular, we must have $\text{rad}(q)|P$. \\
Put $P':= P/(P,q)$. It now follows that $f(n) = \chi(n) = 0$ whenever $p|q$, and thus $f(n) = \chi(n)$ for all $(n,2P') = 1$. If $q$ is even then we are done since then $f(n) = \chi(n)$ for all $(n,P') = 1$, and we can then take $f = \chi \chi_0^{(P')}$, where $\chi_0^{(P')}$ is the principal character modulo $P'$. It thus remains to consider $q$ odd, and by this same argument it suffices to check that $f(2) = \chi(2)$. \\
Combining Proposition \ref{POSCOEFFS} with \eqref{RECURM}, again dropping all choices of $d$ except $d = 1$, we see that for all sufficiently large $H$,
\begin{equation} \label{2ADIC}
|\tilde{G}(2)|\sum_{k \geq 1} \sum_{g|\text{rad}(q)} \mu(g)g^{-2}\left\|\frac{Hg}{2^k}\right\| = O(1).
\end{equation}
We now fix $K$ large, and let $\{k_j\}_{j\leq J}$ be the set of $k_j \leq K$ such that $2^{K-k_j} \equiv 1 (\text{rad}(q))$ (that is, take $k_j := K-j\phi(\text{rad}(q))$ for $j \leq J$). Put $H := 2^K/\text{rad}(q)$. Then observe that 
\begin{equation*}
\left\|Hg/2^{k_j}\right\| = \left\|g/\text{rad}(q)\right\| = \begin{cases} 0 &\text{ if $g = \text{rad}(q)$} \\ g/\text{rad}(q) &\text{ otherwise}. \end{cases}
\end{equation*}
Thus, for each $j \leq J$ we have
\begin{equation*}
\sum_{g|\text{rad}(q)} \mu(g)g^{-2}\left\|\frac{Hg}{2^{k_j}}\right\| = \frac{1}{\text{rad}(q)}\sum_{g|\text{rad}(q) \atop g \neq \text{rad}(q)} \mu(g)g^{-1} = \frac{1}{\text{rad}(q)} \left(\frac{\phi(q)}{q} - \frac{\mu(\text{rad}(q))}{\text{rad}(q)}\right) \neq 0.
\end{equation*}
It follows from this and \eqref{2ADIC} that
\begin{equation*}
|\tilde{G}(2)|J\asymp |\tilde{G}(2)| \sum_{j \leq J} \sum_{g|\text{rad}(q)} \mu(g)g^{-2}\left\|\frac{Hg}{2^{k_j}}\right\| = O(1),
\end{equation*}
and since $K$ (and thus $J$) can be taken arbitrarily large, it must follow that $\tilde{G}(2) = 0$. Hence, in all cases we have $F(2)= 1$, i.e., $f(2) = \chi(2)$ as well. We can thus argue as in the case that $q$ is even and conclude that $f(n) = \chi(n)\chi_0^{(P')}(n)$ for all $n$. This completes the proof of Theorem \ref{CHUDAKOV}.
\end{proof}
\section{On a Variant of Cohn's Conjecture} \label{SECCOHN}
Let $x \geq 3$ be large let $H$ be a large but fixed real number. We observe first that for $g = f$ and $g = \chi$, we have
\begin{align}
\sum_{m \leq x} m^{-1}&\left|\sum_{m < n \leq m+H} g(n)\right|^2 = \sum_{m \leq x} \frac{1}{m}\sum_{|h| \leq H} \sum_{m < n,n+h \leq m+H} g(n)\overline{g(n+h)} \nonumber\\
&= \sum_{|h| \leq H} \sum_{m \leq x} \sum_{m < n,n+h \leq m+H} \frac{g(n)\overline{g(n+h)}}{n} \left(1+O\left(H/n\right)\right) \nonumber\\
&= \sum_{|h| \leq H} \sum_{n \leq x} |\{m\leq x : m < n,n+h \leq m+H\}| \frac{g(n)\overline{g(n+h)}}{n}\left(1+O\left(H/n\right)\right) \nonumber\\
&= \sum_{|h| \leq H} (H-|h|) \sum_{n \leq x} \frac{g(n)\overline{g(n+h)}}{n} + O(H^3).\label{WTD}
\end{align}
By a similar (but simpler) argument we also have
\begin{equation} \label{UNWTD}
\sum_{m \leq x} \left|\sum_{m<n \leq m+H} g(n)\right|^2 = \sum_{|h|\leq H} (H-|h|) \sum_{n \leq x} g(n)\overline{g(n+h)} + O(H^3).
\end{equation}
One deduces from \eqref{ast2} and \eqref{UNWTD} that
\begin{equation} \label{SMM}
\sum_{m \leq x} \left|\sum_{m < n \leq m+H}f(n)\right|^2 = (1+o(1)) \sum_{m \leq x} \left|\sum_{m < n \leq m+H} \chi(n)\right|^2 + O(H^3).
\end{equation}
Put $S_f(m;H) := \sum_{m < n \leq m+H} f(m)$. By partial summation,
\begin{align*}
\sum_{m \leq x} m^{-1} \left|S_f(m;H)\right|^2 &= \int_1^x t^{-1}d\left\{\sum_{m \leq t} \left|S_f(m;H)\right|^2\right\} \\
&= x^{-1}\sum_{m \leq x} \left|S_f(m;H)\right|^2 + \int_1^x \frac{dt}{t^2}\left(\sum_{m \leq t} |S_f(m;H)|^2\right).
\end{align*}
Applying \eqref{SMM} for each $t \leq x$, we have
\begin{align*}
\sum_{m \leq x} m^{-1}\left|S_f(m;H)\right|^2 &= (1+o(1))\left(x^{-1}\sum_{m \leq x}|S_{\chi}(m;H)|^2 + \int_1^x \frac{dt}{t^2} \left(\sum_{m \leq t} |S_{\chi}(m;H)|^2\right)\right) + O(H^3) \\
&= (1+o(1))\sum_{m \leq x} m^{-1} |S_{\chi}(m;H)|^2 + O(H^3),
\end{align*}
and in light of \eqref{WTD} it follows that
\begin{equation*}
\sum_{|h| \leq H}(H-|h|) \sum_{n \leq x} \frac{f(n)\overline{f(n+h)}}{n} = (1+o(1))\sum_{|h| \leq H} (H-|h|)\sum_{n \leq x} \frac{\chi(n)\overline{\chi(n+h)}}{n} + O(H^3).
\end{equation*}
Now, since $q$ is fixed we may suppose that $x$ is sufficiently large and that $H= mq$, for some fixed positive integer $m$. Then each of the short sums with $g = \chi$ in \eqref{WTD} is 0, and as such we have
\begin{equation*}
\sum_{|h| \leq H}(H-|h|) \sum_{n \leq x} \frac{f(n)\overline{f(n+h)}}{n} \ll H^3.
\end{equation*}
On the other hand, separating the $h =0$ term from the remaining quantity, we have
\begin{equation*}
\left|\sum_{1 \leq |h| \leq H} (H-|h|) \sum_{n \leq x} \frac{f(n)\overline{f(n+h)}}{n}\right| \geq H\log x + O(H^3).
\end{equation*}
As such, there must be some $h = h_x$ such that $\left|\sum_{n \leq x} \frac{f(n)\overline{f(n+h)}}{n}\right| \gg q^{-1}\log x$. By Theorem \ref{TAOBINTHM}, it follows that there is a primitive Dirichlet character $\chi_x$ of bounded (in terms of $q$) conductor and a real number $t_x$ with $|t_x| \ll_q x$ such that $D(f,\chi_xn^{it_x};x) \ll_q 1$. 
By Lemma \ref{ELLLEM}, we can find some fixed $t \in \mb{R}$ such that $D(f,\chi n^{it}; x) \ll_q 1$ as $x \ra \infty$. \\
Applying Corollary 3.4 of \cite{Klu} to both $f$ and $\chi$ (and recalling that $2\nmid q$), we have
\begin{align*}
x^{-1}\sum_{n \leq x} f(n)\overline{f(n+1)}  &= (1+o(1))\frac{\mu(m)}{m}\prod_{p \geq 1\atop p \nmid m} \left(2\left(1-\frac{1}{p}\right)\left(\sum_{k\geq 0}\frac{\text{Re}(f(p^k)\overline{\chi'(p^k)p^{ikt}})}{p^k}\right)-1\right); \\
x^{-1}\sum_{n \leq x} \chi(n)\overline{\chi(n+1)}  &= (1+o(1))\frac{\mu(q)}{q}.
\end{align*}
As such, we must have
\begin{equation} \label{MQ}
\frac{\mu(m)}{m}\prod_{p \geq 1\atop p \nmid m} \left(2\text{Re}\left(\left(1-\frac{1}{p}\right)\left(\sum_{k\geq 0}\frac{f(p^k)\overline{\chi'(p^k)p^{ikt}}}{p^k}\right)\right)-1\right) = \frac{\mu(q)}{q}.
\end{equation}
Now if we let $h = P$, where $P$ is a prime not dividing $m$, and apply Theorem 1.5 there in its place, the right side of this last equation is the same, while the local factor with $p = P$ is the only thing that changes. In particular, we have
\begin{align*}
&2\text{Re}\left(\left(1-\frac{1}{P}\right) \left(1+P^{-1}f(P)\overline{\chi'(P)P^{it}} + \sum_{k \geq 2} \frac{f(P^k)\overline{\chi(P^k)P^{ikt}}}{P^k}\right)\right) -1 \\
&= 1-2P^{-2} + 2\text{Re}\left(\left(1-\frac{1}{P}\right)\left(\overline{f(P)}\chi(P)P^{it}\sum_{k \geq 2} \frac{f(P^k)\overline{\chi(P^k)P^{ikt}}}{P^k}\right)\right).
\end{align*}
Rearranging this identity and manipulating further yields	
\begin{equation*}
\text{Re}\left(\left(\overline{f(P)}\chi(P)P^{it}-1\right) \sum_{k \geq 2} \frac{f(P^k)\overline{\chi(P^k)P^{ikt}}}{P^k}\right) = \frac{1}{P}\text{Re}\left(f(P)\overline{\chi(P)P^{it}}-1\right).
\end{equation*}
Note, however, that the sum over $k$ is bounded at most by $P^{-2}(1-1/P)^{-1}$, so unless $f(P)\overline{\chi'(P)P^{it}} = 1$ or $P = 2$, this is a contradiction. It follows that $f(p) = \chi'(p)p^{it}$ for each $p \nmid 2m$.  If $P = 2$ then either $f(2) = 2^{it}\chi'(2)$, or else $f(2^k) = \chi(2^k)2^{ikt}$ for $k \geq 2$. A similar argument shows that $f(p^k) = \chi'(p^k)p^{itk}$ as well, for $p \nmid 2m$, and for $p = 2$ in the first case listed. Thus, we assume that $f(2^k) = 2^{ikt}\chi'(2)^k$ for $k \geq 2$. \\
It thus remains to show that $f(2) = 2^{it}\chi'(2)$ as well in this case. To see this, we put $g(n) := \chi'(n)n^{it}$ and note using the above that 
\begin{align*}
\sum_{n \leq x} f(n)\bar{f(n+4)} = \left(f(2)2^{-it}\bar{\chi'(2)}-1\right)\sum_{n \leq x \atop 2|| n} g(n)\bar{g(n+4)} + \sum_{n \leq x} g(n)\bar{g(n+4)}.
\end{align*}
Applying Corollary 3.4 of \cite{Klu} once again and using \eqref{ast2}, we note that 
\begin{equation}\label{CLOSE4}
\sum_{n \leq x} g(n)\bar{g(n+4)} - \sum_{n \leq x} f(n)\bar{f(n+4)}= o(x),
\end{equation}
as $x \ra \infty$. Since we also have
\begin{align*}
\sum_{n \leq x \atop 2||n} g(n)\bar{g(n+4)} &= \sum_{m \leq x/2 \atop 2\nmid m} g(m)\bar{g(m+2)},
\end{align*}
it follows from \eqref{CLOSE4} that
\begin{align}
\left(f(2)2^{-it}\bar{\chi'(2)} -1\right)\sum_{m \leq x/2 \atop 2 \nmid m} g(m)\bar{g(m+2)} = o(x). \label{END}
\end{align}
Note that if we define an arithmetic function $h$ by letting $h(2) = 0$, $h(2^k) = g(2^k)$ for $k \geq 2$ and $h(p^k) = g(p^k)$ for all $p \geq 3$ and $k \geq 1$ and extend it multiplicatively then the second sum is simply $\sum_{m \leq x/2} h(m)\bar{h(m+2)}$, with $h$ satisfying $\mb{D}(h,n^{it}\chi;\infty) < \infty$. Dividing this last equation by $x/2$ and taking $x \ra \infty$, the resulting limit on the left side is real and positive by Theorem 1.5 of \cite{Klu}. It follows from \eqref{END} that $f(2)=2^{it}\chi'(2)$, as required.
Thus $f(n) = \chi'(n)n^{it}$, and it follows furthermore that $m =q$ from \eqref{MQ}. Thus, $\chi'$ is a character of order $q$. This completes the proof.

\end{document}